\documentclass[12pt]{amsart}
\usepackage[dvipsnames]{xcolor}
\usepackage{tikz}
\usepackage{graphicx}
\usepackage{amssymb, amsthm, amsmath}
\usepackage{etoolbox}
\apptocmd{\sloppy}{\hbadness 10000\relax}{}{}
\apptocmd{\sloppy}{\vbadness 10000\relax}{}{}
\usepackage[justification=justified]{caption}
\usepackage{hyperref}
\usepackage[letterpaper,margin=1.1in]{geometry}

\usepackage{longtable}
\usepackage{needspace}
\usepackage{booktabs}
\usepackage{array}
\usepackage{enumitem}

\definecolor{hc}{RGB}{0,128,128}
\definecolor{vc}{RGB}{142,69,133}

\newtheorem{theorem}{Theorem}[section]
\newtheorem{proposition}[theorem]{Proposition}
\newtheorem{lemma}[theorem]{Lemma}
\newtheorem{corollary}[theorem]{Corollary}

\theoremstyle{definition}
\newtheorem{definition}[theorem]{Definition}

\theoremstyle{remark}

\numberwithin{equation}{section}
\numberwithin{figure}{section}
\numberwithin{table}{section}

\allowdisplaybreaks[1]
\newcommand{\F}{F_8}
\newcommand{\KK}{\mathcal K}
\newcommand{\omicron}{o}
\newcommand{\st}[1]{\mathrm{#1}}
\newcommand{\Ppi}{P_\pi}
\newcommand{\Prho}{P_\rho}
\newcommand{\Pchi}{P_\chi}

\title[Optimal packing of eight squares]{An exact formula for optimal axis-parallel packing of eight squares}
\date{October 20, 2026}
\author{Matthew Badger}
\address{Department of Mathematics\\ University of Connecticut\\ Storrs, CT 06269-1009, USA}
\email{matthew.badger@uconn.edu}
\begin{document}
\begin{abstract}
We determine the smallest square container for eight axis-parallel squares
of prescribed side lengths. The optimal side length is the minimum of the
costs of 23 explicit packing strategies, and each cost is a maximum of sums
of side lengths. No minimum of fewer convex functions represents it. Three
strategies suffice for seven squares; the jump to 23 for eight squares
witnesses the beginning of a combinatorial explosion. We prove the formula by combining
the constructions with geometric lower bounds derived from disjoint
coordinate projections, cornering arguments, and a weighted area comparison.
\end{abstract}

\maketitle

\setcounter{tocdepth}{1}
\tableofcontents

\section{Introduction}\label{sec:introduction}
Container problems ask for the smallest container of a given shape that
holds prescribed convex bodies with pairwise disjoint interiors. Packing squares in a
square already leads to difficult questions. Moon and Moser~\cite{MoonMoser67}
showed that axis-parallel squares of total area at most $1/2$ always fit in a
unit square, and that $1/2$ is sharp.
Erd\H{o}s and Graham~\cite{ErdosGraham75} showed that unit squares, suitably
tilted, can be packed into a square of large side $a$ leaving uncovered area
$O(a^{7/11})$. Friedman's survey~\cite{Friedman09} describes the smallest squares
known to contain $n$ unit squares when rotations are allowed, and the
cases for which optimality has been proved. For example,
Stromquist~\cite{Stromquist03} proved that ten unit squares require a square
of side $3+1/\sqrt2$. A computer-assisted proof of the optimality of Walter
Trump's 1979 packing of eleven unit squares has recently been
announced~\cite{ElevenSquares26}, together with a Lean
formalization~\cite{ElevenSquaresFormalized26}.

For items of fixed orientation, Murata, Fujiyoshi, Nakatake, and
Kajitani~\cite{Murata96} encode placements of rectangles by sequence-pairs,
and Fekete and Schepers~\cite{FeketeSchepers04} characterize feasible
orthogonal packings in every dimension by packing classes. Encodings of this
kind underlie exact search methods for individual instances. For example,
Huang and Korf~\cite{HuangKorf13} find all minimum-area enclosing rectangles
for a given set of rectangles, and extend the known solutions of the
benchmark of consecutive squares $1,2,\ldots,n$ to $n=32$.
When the number of items is part of the input, deciding whether prescribed
squares fit without rotations into a prescribed square is strongly
NP-complete~\cite{Leung90}.

Here the side lengths are arbitrary, the orientations are fixed, and the
goal is an exact formula valid for every ordered side vector. Given a
nonincreasing sequence $s=(s_0,s_1,\ldots)$ of nonnegative numbers, let
$L^*(s)$ be the least side of a square that contains axis-parallel squares
$Q_0,Q_1,\ldots$ of sides $s_0,s_1,\ldots$ with pairwise disjoint interiors. Sides of length
zero signify absent squares. The value $L^*(s)$ is finite exactly when
$s\in\ell_2$, and we write $\ell_2^\downarrow$ for the cone of nonincreasing
nonnegative sequences in $\ell_2$. We identify a finite nonincreasing
nonnegative tuple with the sequence in $\ell_2^\downarrow$ obtained by
appending zeros. For $n\ge1$ and $s\in\ell_2^\downarrow$, we write
$L_n^*(s)=L^*(s_0,\ldots,s_{n-1},0,0,\ldots)$ for the optimum for the $n$ largest squares. In~\cite{BadgerSeven} we gave
exact formulas for $L_n^*$ when $n\le7$, together with methods for proving
lower bounds. In particular, for every $s\in\ell_2^\downarrow$,
\begin{align}
L_6^*(s)&=\max(s_0+s_1,\ s_2+s_3+s_4,\ s_1+s_4+s_5),\label{eq:intro-six}\\
L_7^*(s)&=\max\bigl(L_6^*(s),\min\{\max(s_3+s_4+s_5+s_6,\ s_1+s_2+s_5),\notag\\
&\hphantom{{}=\max\bigl(L_6^*(s),\min\{}\max(s_0+s_5+s_6,\ s_1+s_3+s_6),\ s_1+s_4+s_5+s_6\}\bigr).\label{eq:intro-seven}
\end{align}

A strategy for $n$ squares prescribes, for each pair of squares, at least
one axis on which they are separated and their order along each such axis.
Its cost is the greatest side sum of a set of squares that the strategy
orders linearly along one axis, and $L_n^*(s)$ is the minimum of the costs
over all strategies. A single strategy suffices for $n=6$. For $n=7$, three
strategies suffice; we call them $\st{top}$, $\st{center}$, and
$\st{right}$, after the position of $Q_6$, and their costs are the three
entries of the minimum in~\eqref{eq:intro-seven}, each combined with
$L_6^*$ by a maximum. For eight squares we prove the following.

\begin{theorem}\label{thm:main}
For every $s\in\ell_2^\downarrow$,
\[
L_8^*(s)=\max\bigl(L_7^*(s),\ \min\{\alpha(s),\beta(s),\ldots,\psi(s)\}\bigr)
=\min_\omega\max\bigl(L_6^*(s),\omega(s)\bigr),
\]
where $\alpha,\beta,\ldots,\psi$ are the twenty-three maxima of side sums
in Table~\ref{tab:pieces} and $\omega$ ranges over them. For each $\omega$,
the function $\ell_\omega=\max(L_6^*,\omega)$ is the cost of the strategy
$\sigma_\omega$ defined by the coordinates in the row of $\omega$ in
Table~\ref{tab:coords8}.

Moreover, if $h_1,\ldots,h_m$ are real-valued convex functions on
$\ell_2^\downarrow$ with $L_8^*(s)=\min_kh_k(s)$ for every
$s\in\ell_2^\downarrow$, then $m\ge23$. In particular, every complete family
of strategies for eight squares has at least twenty-three members.
\end{theorem}
\begingroup
\small
\renewcommand{\arraystretch}{1.12}
\begin{longtable}{@{}r@{${}={}$}l@{}}
\caption{The twenty-three maxima of side sums in Theorem~\ref{thm:main}.}\label{tab:pieces}\\
\toprule
\endfirsthead
\multicolumn{2}{c}{\tablename~\thetable\ (continued)}\\[2pt]
\toprule
\endhead
\bottomrule
\endlastfoot
$\alpha$&$\max(s_0+s_4+s_7,\ s_0+s_5+s_6,\ s_1+s_2+s_7,\ s_1+s_3+s_5)$\\
$\beta$&$\max(s_0+s_3+s_7,\ s_0+s_5+s_6,\ s_1+s_3+s_6,\ s_2+s_5+s_6+s_7)$\\
$\gamma$&$\max(s_0+s_3+s_7,\ s_0+s_5+s_6,\ s_1+s_2+s_6)$\\
$\delta$&$\max(s_0+s_4+s_5,\ s_1+s_2+s_7,\ s_1+s_3+s_6)$\\
$\varepsilon$&$\max(s_0+s_5+s_6,\ s_1+s_3+s_4,\ s_2+s_4+s_6+s_7)$\\
$\zeta$&$\max(s_0+s_6+s_7,\ s_1+s_2+s_4,\ s_3+s_4+s_5+s_6)$\\
$\eta$&$\max(s_0+s_3+s_6,\ s_1+s_2+s_7)$\\
$\theta$&$\max(s_0+s_2+s_7,\ s_0+s_5+s_6,\ s_1+s_3+s_6)$\\
$\iota$&$\max(s_0+s_4+s_7,\ s_0+s_5+s_6,\ s_1+s_3+s_5,\ s_2+s_5+s_6+s_7)$\\
$\kappa$&$\max(s_0+s_5+s_6,\ s_1+s_2+s_4,\ s_3+s_4+s_6+s_7)$\\
$\lambda$&$\max(s_0+s_4+s_5,\ s_1+s_3+s_6,\ s_2+s_5+s_6+s_7)$\\
$\mu$&$\max(s_0+s_5+s_6,\ s_1+s_3+s_5,\ s_1+s_4+s_6+s_7)$\\
$\nu$&$\max(s_0+s_6+s_7,\ s_1+s_3+s_7,\ s_1+s_4+s_5+s_6)$\\
$\xi$&$\max(s_1+s_3+s_5,\ s_1+s_3+s_6+s_7,\ s_2+s_4+s_5+s_6)$\\
$\omicron$&$\max(s_0+s_4+s_7,\ s_1+s_2+s_5,\ s_3+s_4+s_5+s_6)$\\
$\pi$&$\max(s_0+s_5+s_6,\ s_1+s_3+s_6+s_7)$\\
$\rho$&$\max(s_0+s_5+s_6,\ s_1+s_3+s_6,\ s_1+s_4+s_5+s_7)$\\
$\sigma$&$\max(s_1+s_4+s_6+s_7,\ s_2+s_3+s_5+s_6)$\\
$\tau$&$\max(s_1+s_4+s_5+s_6,\ s_2+s_3+s_6+s_7)$\\
$\upsilon$&$\max(s_1+s_5+s_6+s_7,\ s_2+s_3+s_4+s_6)$\\
$\varphi$&$\max(s_1+s_2+s_5,\ s_1+s_2+s_6+s_7,\ s_3+s_4+s_5+s_6)$\\
$\chi$&$\max(s_0+s_5+s_6+s_7,\ s_1+s_3+s_6)$\\
$\psi$&$s_1+s_4+s_5+s_6+s_7$\\
\end{longtable}
\endgroup

For each expression $\omega$ in Table~\ref{tab:pieces}, the coordinates in
the row of $\omega$ in Table~\ref{tab:coords8} give a packing in every square
of side at least the strategy cost $\ell_\omega$; Figure~\ref{fig:strategies}
illustrates them. This proves the upper bound in the formula
of Theorem~\ref{thm:main}. For the lower bound we prove eight new
inequalities $L_8^*(s)\ge b_i(s)$, $i=1,\ldots,8$, expressed using minima
and maxima of side sums of $s_0,\ldots,s_7$, and use one more, a column
bound from~\cite{BadgerSeven}. Most of the new bounds follow from the
chain and cornering arguments of~\cite{BadgerSeven}, and one from a weighted
comparison of areas. A comparison of side sums then shows that the maximum
of $L_7^*$, the column bound, and $b_1,\ldots,b_8$ is at least the
right-hand side of the formula. This step compares
explicit formulas and requires no classification of optimal packings. The
bound $m\ge23$ in Theorem~\ref{thm:main} follows from twenty-three
side vectors $p_\omega$, one for each $\omega$, at which $L_8^*$ fails to be
subadditive for every pair of distinct vectors:
\[
L_8^*(p_\omega+p_{\omega'})>L_8^*(p_\omega)+L_8^*(p_{\omega'})\quad\text{for all }\omega\ne\omega'.
\]

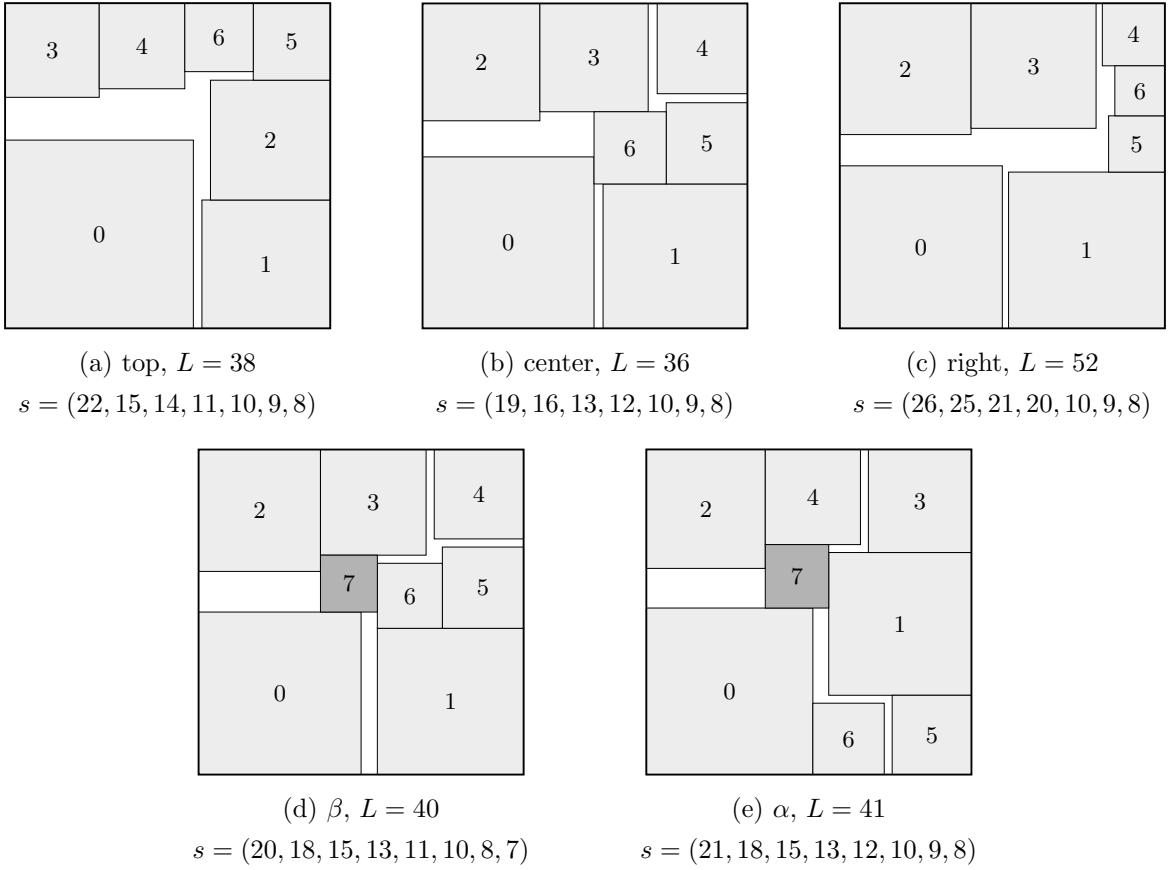
\begin{figure}[!htbp]\centering
\begin{minipage}[t]{0.31\textwidth}\centering
\begin{tikzpicture}[x=0.1132cm,y=0.1132cm]
\useasboundingbox (0,0) rectangle (38,38);
\filldraw[fill=black!7,draw=black,line width=.35pt] (0,0) rectangle (22,22);
\node[font=\scriptsize,inner sep=0pt] at (11,11) {$0$};
\filldraw[fill=black!7,draw=black,line width=.35pt] (23,0) rectangle (38,15);
\node[font=\scriptsize,inner sep=0pt] at (30.5,7.5) {$1$};
\filldraw[fill=black!7,draw=black,line width=.35pt] (24,15) rectangle (38,29);
\node[font=\scriptsize,inner sep=0pt] at (31,22) {$2$};
\filldraw[fill=black!7,draw=black,line width=.35pt] (0,27) rectangle (11,38);
\node[font=\scriptsize,inner sep=0pt] at (5.5,32.5) {$3$};
\filldraw[fill=black!7,draw=black,line width=.35pt] (11,28) rectangle (21,38);
\node[font=\scriptsize,inner sep=0pt] at (16,33) {$4$};
\filldraw[fill=black!7,draw=black,line width=.35pt] (29,29) rectangle (38,38);
\node[font=\scriptsize,inner sep=0pt] at (33.5,33.5) {$5$};
\filldraw[fill=black!7,draw=black,line width=.35pt] (21,30) rectangle (29,38);
\node[font=\scriptsize,inner sep=0pt] at (25,34) {$6$};
\draw[line width=.7pt] (0,0) rectangle (38,38);
\end{tikzpicture}\par
\smallskip{\footnotesize (a) $\st{top}$, $L=38$\\ $s=(22,15,14,11,10,9,8)$}
\end{minipage}
\hfill
\begin{minipage}[t]{0.31\textwidth}\centering
\begin{tikzpicture}[x=0.1194cm,y=0.1194cm]
\useasboundingbox (0,0) rectangle (36,36);
\filldraw[fill=black!7,draw=black,line width=.35pt] (0,0) rectangle (19,19);
\node[font=\scriptsize,inner sep=0pt] at (9.5,9.5) {$0$};
\filldraw[fill=black!7,draw=black,line width=.35pt] (20,0) rectangle (36,16);
\node[font=\scriptsize,inner sep=0pt] at (28,8) {$1$};
\filldraw[fill=black!7,draw=black,line width=.35pt] (0,23) rectangle (13,36);
\node[font=\scriptsize,inner sep=0pt] at (6.5,29.5) {$2$};
\filldraw[fill=black!7,draw=black,line width=.35pt] (13,24) rectangle (25,36);
\node[font=\scriptsize,inner sep=0pt] at (19,30) {$3$};
\filldraw[fill=black!7,draw=black,line width=.35pt] (26,26) rectangle (36,36);
\node[font=\scriptsize,inner sep=0pt] at (31,31) {$4$};
\filldraw[fill=black!7,draw=black,line width=.35pt] (27,16) rectangle (36,25);
\node[font=\scriptsize,inner sep=0pt] at (31.5,20.5) {$5$};
\filldraw[fill=black!7,draw=black,line width=.35pt] (19,16) rectangle (27,24);
\node[font=\scriptsize,inner sep=0pt] at (23,20) {$6$};
\draw[line width=.7pt] (0,0) rectangle (36,36);
\end{tikzpicture}\par
\smallskip{\footnotesize (b) $\st{center}$, $L=36$\\ $s=(19,16,13,12,10,9,8)$}
\end{minipage}
\hfill
\begin{minipage}[t]{0.31\textwidth}\centering
\begin{tikzpicture}[x=0.0827cm,y=0.0827cm]
\useasboundingbox (0,0) rectangle (52,52);
\filldraw[fill=black!7,draw=black,line width=.35pt] (0,0) rectangle (26,26);
\node[font=\scriptsize,inner sep=0pt] at (13,13) {$0$};
\filldraw[fill=black!7,draw=black,line width=.35pt] (27,0) rectangle (52,25);
\node[font=\scriptsize,inner sep=0pt] at (39.5,12.5) {$1$};
\filldraw[fill=black!7,draw=black,line width=.35pt] (0,31) rectangle (21,52);
\node[font=\scriptsize,inner sep=0pt] at (10.5,41.5) {$2$};
\filldraw[fill=black!7,draw=black,line width=.35pt] (21,32) rectangle (41,52);
\node[font=\scriptsize,inner sep=0pt] at (31,42) {$3$};
\filldraw[fill=black!7,draw=black,line width=.35pt] (42,42) rectangle (52,52);
\node[font=\scriptsize,inner sep=0pt] at (47,47) {$4$};
\filldraw[fill=black!7,draw=black,line width=.35pt] (43,25) rectangle (52,34);
\node[font=\scriptsize,inner sep=0pt] at (47.5,29.5) {$5$};
\filldraw[fill=black!7,draw=black,line width=.35pt] (44,34) rectangle (52,42);
\node[font=\scriptsize,inner sep=0pt] at (48,38) {$6$};
\draw[line width=.7pt] (0,0) rectangle (52,52);
\end{tikzpicture}\par
\smallskip{\footnotesize (c) $\st{right}$, $L=52$\\ $s=(26,25,21,20,10,9,8)$}
\end{minipage}
\par\bigskip
\begin{minipage}[t]{0.31\textwidth}\centering
\begin{tikzpicture}[x=0.1075cm,y=0.1075cm]
\useasboundingbox (0,0) rectangle (40,40);
\filldraw[fill=black!7,draw=black,line width=.35pt] (0,0) rectangle (20,20);
\node[font=\scriptsize,inner sep=0pt] at (10,10) {$0$};
\filldraw[fill=black!7,draw=black,line width=.35pt] (22,0) rectangle (40,18);
\node[font=\scriptsize,inner sep=0pt] at (31,9) {$1$};
\filldraw[fill=black!7,draw=black,line width=.35pt] (0,25) rectangle (15,40);
\node[font=\scriptsize,inner sep=0pt] at (7.5,32.5) {$2$};
\filldraw[fill=black!7,draw=black,line width=.35pt] (15,27) rectangle (28,40);
\node[font=\scriptsize,inner sep=0pt] at (21.5,33.5) {$3$};
\filldraw[fill=black!7,draw=black,line width=.35pt] (29,29) rectangle (40,40);
\node[font=\scriptsize,inner sep=0pt] at (34.5,34.5) {$4$};
\filldraw[fill=black!7,draw=black,line width=.35pt] (30,18) rectangle (40,28);
\node[font=\scriptsize,inner sep=0pt] at (35,23) {$5$};
\filldraw[fill=black!7,draw=black,line width=.35pt] (22,18) rectangle (30,26);
\node[font=\scriptsize,inner sep=0pt] at (26,22) {$6$};
\filldraw[fill=black!30,draw=black,line width=.35pt] (15,20) rectangle (22,27);
\node[font=\scriptsize,inner sep=0pt] at (18.5,23.5) {$7$};
\draw[line width=.7pt] (0,0) rectangle (40,40);
\end{tikzpicture}\par
\smallskip{\footnotesize (d) $\beta$, $L=40$\\ $s=(20,18,15,13,11,10,8,7)$}
\end{minipage}
\qquad
\begin{minipage}[t]{0.31\textwidth}\centering
\begin{tikzpicture}[x=0.1049cm,y=0.1049cm]
\useasboundingbox (0,0) rectangle (41,41);
\filldraw[fill=black!7,draw=black,line width=.35pt] (0,0) rectangle (21,21);
\node[font=\scriptsize,inner sep=0pt] at (10.5,10.5) {$0$};
\filldraw[fill=black!7,draw=black,line width=.35pt] (23,10) rectangle (41,28);
\node[font=\scriptsize,inner sep=0pt] at (32,19) {$1$};
\filldraw[fill=black!7,draw=black,line width=.35pt] (0,26) rectangle (15,41);
\node[font=\scriptsize,inner sep=0pt] at (7.5,33.5) {$2$};
\filldraw[fill=black!7,draw=black,line width=.35pt] (28,28) rectangle (41,41);
\node[font=\scriptsize,inner sep=0pt] at (34.5,34.5) {$3$};
\filldraw[fill=black!7,draw=black,line width=.35pt] (15,29) rectangle (27,41);
\node[font=\scriptsize,inner sep=0pt] at (21,35) {$4$};
\filldraw[fill=black!7,draw=black,line width=.35pt] (31,0) rectangle (41,10);
\node[font=\scriptsize,inner sep=0pt] at (36,5) {$5$};
\filldraw[fill=black!7,draw=black,line width=.35pt] (21,0) rectangle (30,9);
\node[font=\scriptsize,inner sep=0pt] at (25.5,4.5) {$6$};
\filldraw[fill=black!30,draw=black,line width=.35pt] (15,21) rectangle (23,29);
\node[font=\scriptsize,inner sep=0pt] at (19,25) {$7$};
\draw[line width=.7pt] (0,0) rectangle (41,41);
\end{tikzpicture}\par
\smallskip{\footnotesize (e) $\alpha$, $L=41$\\ $s=(21,18,15,13,12,10,9,8)$}
\end{minipage}
\caption{Top row: packings of seven squares following $\st{top}$,
$\st{center}$, and $\st{right}$ of~\cite[Table~5.1]{BadgerSeven}, each with
side $L_7^*(s)$. Bottom row: packings of eight squares following
$\sigma_\beta$ and $\sigma_\alpha$, each with side $L_8^*(s)$; $Q_7$ is
shaded.}
\label{fig:insert}
\end{figure}

An optimal strategy for eight squares need not restrict to an optimal
strategy for the seven largest squares. Deleting $Q_7$ from $\sigma_\alpha$
leaves a strategy for $Q_0,\ldots,Q_6$ with cost
$\max(L_6^*,\ s_0+s_5+s_6,\ s_1+s_3+s_5)$. At the side vector
$s=(21,18,15,13,12,10,9,8)$ of Figure~\ref{fig:insert}(e), this cost is
$41=L_8^*(s)$, whereas $L_7^*(s)=40$: the indices $5,1,3$ form a vertical
chain of $\sigma_\alpha$ with side sum $41$. In this sense the passage from
seven to eight squares is nonlocal: an optimal arrangement of eight squares
need not extend an optimal arrangement of the seven largest. By contrast,
deleting $Q_7$ from $\sigma_\beta$ leaves the arrangement $\st{center}$
(Figure~\ref{fig:insert}(b),(d)). Section~\ref{sec:eight} shows that the
restrictions of thirteen of the twenty-three strategies have cost functions
different from all three seven-square cost functions.

By~\cite[Theorem~6.2]{BadgerSeven} and Theorem~\ref{thm:main}, the least
sizes of complete families of strategies for six, seven, and eight squares
are $1$, $3$, and $23$. The successive ratios are $3$ and $23/3$, and the
second exceeds the first by the factor $23/9$. The lower-bound proofs grow
as well: the proof for seven squares in~\cite{BadgerSeven} combines $L_6^*$
with three inequalities, and the proof here combines $L_7^*$ with eight new
inequalities and the column bound of~\cite{BadgerSeven}. We do
not know the least size for nine squares. If the ratio stayed at $23/3$, the
least size would be about $176$. If the ratio grew again by
the factor $23/9$, it would be about $19.6$, giving about $450$. Both estimates are speculative, but they suggest that an exact
formula for $L_9^*$ of the same type would be a minimum of hundreds of
maxima of side sums.

For general $n$, a complete family of $m(n)$
strategies decides whether $n$ prescribed squares fit into a prescribed
square by evaluating $m(n)$ strategy costs, each the largest side sum of a
chain in a pair of partial orders. Each strategy can be recorded by $O(n^2)$
order relations, and its cost can be computed in polynomial time: for either
partial order, with vertex weights $s_i$, the largest side sum of a chain is
the weight of a longest path in its directed graph, computed by the usual
recurrence over predecessors. If $m(n)$ were bounded by a polynomial in
$n$, these families would serve as advice of polynomial size for the
strongly NP-complete problem of~\cite{Leung90}. Then
$\mathrm{NP}\subseteq\mathrm{P/poly}$, and the polynomial hierarchy would
collapse to its second level by the Karp--Lipton
theorem~\cite{KarpLipton82}. Thus, unless this collapse occurs, the least size of a complete family of strategies for
$n$ squares is not bounded above by any polynomial in $n$. In the other
direction, the sequence-pair
families of Silvanus and Vygen~\cite{SilvanusVygen20} give complete families
of size $O\bigl(n!\,((11+5\sqrt5)/2)^n/n^6\bigr)$; see~\cite{BadgerSeven}
for a comparison of their notion of completeness with ours.

The rest of the paper is organized as follows. Section~\ref{sec:prelim}
collects the definitions and results of~\cite{BadgerSeven} that we use and
proves two refinements of them. Section~\ref{sec:eight} gives the
twenty-three strategies and their constructions. Section~\ref{sec:lower}
proves the eight new lower bounds. Section~\ref{sec:proof} combines them with
the column bound and $L_7^*$ to prove the matching lower bound by comparing side sums,
completes the proof of Theorem~\ref{thm:main}, and proves the bound
$m\ge23$. The code~\cite{BadgerCode} verifies the constructions in
Table~\ref{tab:coords8}, the packings in the figures, the comparisons
following~\eqref{eq:restrict}, and the computations in
Section~\ref{sec:minimality}.

\section{Preliminaries}\label{sec:prelim}
We recall the conventions and geometric tools of~\cite{BadgerSeven}
without proof, and prove a weighted-chain comparison and a refinement of
its covering-chain lemma.

\subsection{Packings and strategies}
For each positive side $s_i$, write
$Q_i=[x_i,x_i+s_i]\times[y_i,y_i+s_i]$.
These squares form a packing in $[0,L]^2$ exactly when
\begin{equation}\label{eq:contain}
0\le x_i,y_i\le L-s_i
\end{equation}
and, for each pair $i\ne j$, at least one of
\begin{equation}\label{eq:sep}
x_i+s_i\le x_j,\quad x_j+s_j\le x_i,\quad
y_i+s_i\le y_j,\quad y_j+s_j\le y_i.
\end{equation}
\begin{definition}[Strategies and complete families, {\cite[Definition~2.3]{BadgerSeven}}]\label{def:strategy}\label{def:complete}
A \emph{strategy} for $n$ squares is a pair
$\sigma=(\prec_{\mathrm h},\prec_{\mathrm v})$ of strict partial orders on
$\{0,\ldots,n-1\}$ such that every pair of distinct indices is comparable
in at least one order. The relations prescribe horizontal and vertical
separation. A \emph{chain} of $\sigma$ is a nonempty set of indices totally
ordered by either relation. Write $\mathcal C(\sigma)$ for the set of these
chains and define the \emph{cost} of $\sigma$ at a vector
$v\in[0,\infty)^n$ by
$\ell_\sigma(v)=\max_{I\in\mathcal C(\sigma)}\sum_{i\in I}v_i$. For
$s\in\ell_2^\downarrow$, we write $\ell_\sigma(s)$ for its value at
$(s_0,\ldots,s_{n-1})$.
A family $\mathcal S$ is \emph{complete} if
$L_n^*(s)=\min_{\sigma\in\mathcal S}\ell_\sigma(s)$ for every
$s\in\ell_2^\downarrow$.
\end{definition}

\begin{lemma}[Strategy representation, {\cite[Proposition~2.5]{BadgerSeven}}]\label{lem:strategy}
For every $n\ge1$ and every $s\in\ell_2^\downarrow$,
\begin{equation}\label{eq:strategy-formula}
L_n^*(s)=\min_{\sigma\in\mathfrak S_n}\ell_\sigma(s),
\end{equation}
where $\mathfrak S_n$ is the finite set of all strategies. In particular,
$L_n^*(s)$ is a continuous, positively homogeneous, and piecewise linear
function of $(s_0,\ldots,s_{n-1})$.
\end{lemma}
By continuity, a bound $L_n^*\ge f$ or identity $L_n^*=f$, with $f$
continuous, need only be proved for $s_{n-1}>0$.
We therefore assume positive sides in the geometric arguments.

\subsection{Chains, clearing, and pins}
Fix a packing of squares $Q_i$ of sides $s_i$ in $[0,L]^2$. For a set of
indices $I$, call $\sum_{i\in I}s_i$ its \emph{side sum}; the same term
applies to the corresponding collection $\{Q_i:i\in I\}$.
We say that $Q_i$ lies \emph{left of} $Q_j$ when $x_i+s_i\le x_j$, and
\emph{below} $Q_j$ when $y_i+s_i\le y_j$. Two squares are \emph{x-disjoint}
when one lies left of the other, that is, when the interiors of their
horizontal projections are disjoint, and \emph{x-overlapping} otherwise.
Define y-disjointness and y-overlap using vertical projections. Every pair
of squares in a packing is x-disjoint or y-disjoint (possibly both).
A \emph{horizontal chain} is a collection of pairwise x-disjoint squares,
and a \emph{vertical chain} is a collection of pairwise y-disjoint squares.
For positive sides, the left-of and below relations define the
\emph{strategy of the packing}~\cite[Section~2.2]{BadgerSeven}.
Its chains are exactly the nonempty horizontal and vertical chains.

\begin{lemma}[Chain inequality, {\cite[Lemma~3.1]{BadgerSeven}}]\label{lem:chain}
In a packing of side $L$, every horizontal or vertical chain has side sum
at most $L$. In particular, $L^*(s)\ge s_0+s_1$.
\end{lemma}

\begin{lemma}[Column bound, {\cite[Proposition~3.7]{BadgerSeven}}]\label{lem:column}
For every $n\ge6$ and $s\in\ell_2^\downarrow$,
\[
L^*(s)\ge L_n^*(s)\ge\min\bigl(s_1+s_3+s_{n-1},\ P_n(s)\bigr),\quad
\text{where }P_n(s)=s_1+\sum_{i=4}^{n-1}s_i.
\]
\end{lemma}

To prove a lower bound $L^*(s)\ge\Sigma$, suppose a packing has side
$L<\Sigma$. A family of indices $I$, or the corresponding squares,
\emph{clears} $\Sigma$ when $\sum_{i\in I}s_i\ge\Sigma$, and therefore
cannot form a horizontal or vertical chain.
These exclusions force overlaps, and disjointness of square interiors
then gives new chains. For example, suppose $Q_a,Q_b$ are y-disjoint and $\{a,b,e\}$ clears $\Sigma$.
Then $Q_e$ must y-overlap $Q_a$ or $Q_b$, since otherwise the three
squares would form a vertical chain with side sum exceeding $L$.
Two y-overlapping squares are x-disjoint. We call the
forced y-overlap with one of $Q_a,Q_b$ a \emph{vertical pin} of $Q_e$
to the pair $Q_a,Q_b$. Pins to longer chains and horizontal pins are
defined analogously.

We will sometimes translate a square to a wall while keeping the other
squares fixed. For instance, a \emph{blocker to the right} of $Q_i$
is a square that lies right of $Q_i$ and y-overlaps it: moving $Q_i$ right
would eventually meet it. If there is no
such blocker, $Q_i$ can be moved to the right wall without disturbing
the packing. Blockers in other directions are defined analogously.
Reflections and interchanges of the axes move the entire packing.
Because clearing depends only on the side lengths, an exclusion based
on a clearing set remains valid during any motion that preserves a packing.

\begin{lemma}[Weighted-chain comparison]\label{lem:weighted}
Fix a packing of squares $Q_0,\ldots,Q_{n-1}$ with positive sides, and let
$w_0,\ldots,w_{n-1}$ be nonnegative weights, not necessarily in decreasing
order. If $\sum_{i\in I}w_i\le M$ for every horizontal and vertical chain
$\{Q_i:i\in I\}$ of the packing, then squares of side lengths
$w_0,\ldots,w_{n-1}$ can be packed in $[0,M]^2$. In particular,
$\sum_iw_i^2\le M^2$.
\end{lemma}
\begin{proof}
Set $w=(w_0,\ldots,w_{n-1})$, and let $\sigma$ be the strategy of the original
packing. Then $\ell_\sigma(w)\le M$, and~\cite[Lemma~2.4]{BadgerSeven} gives a
packing of squares of side lengths $w_0,\ldots,w_{n-1}$ in
$[0,\ell_\sigma(w)]^2\subseteq[0,M]^2$. Comparing areas gives the final
assertion.
\end{proof}

\subsection{Cornering and covering chains}
\begin{lemma}[Cornering criterion, {\cite[Lemma~3.5]{BadgerSeven}}]\label{lem:pane}
Fix a packing of side $L<\Sigma$. Suppose every triple among
$Q_0,Q_1,Q_2,Q_3$ clears $\Sigma$, as does every triple $\{0,b,j\}$ with
$b\in\{1,2,3\}$ and $j\notin\{0,b\}$ in the full packing.
Then $Q_0$ can be translated to a corner while all other squares remain fixed.
\end{lemma}
\begin{lemma}[Relations forced by four large squares, {\cite[Lemma~3.6]{BadgerSeven}}]\label{lem:fixed-core}
Suppose $Q_0=[0,s_0]^2$ in a packing of $Q_0,\ldots,Q_7$ of side
$L<\Sigma$. Assume every triple among $0,1,2,3$ clears $\Sigma$, as do
\[
\{0,1,m\},\qquad \{0,b,m\},\qquad \{1,b,m\}
\quad\text{for }b=2,3\text{ and }m=4,5.
\]
After possibly interchanging axes, $x_1\ge s_0$ and $y_1<s_0$.
There are distinct $p,c\in\{2,3\}$ such that $p$ x-overlaps $0$,
$c$ x-overlaps $1$ and lies above $1$, and $p$ lies left of $c$.
Each of $4,5$ is x-disjoint from both $p,c$ and y-overlaps both.
These relations hold in the original full packing, before any translation of $Q_1$.
\end{lemma}
\begin{lemma}[Corner frame, {\cite[Lemma~5.2]{BadgerSeven}}]\label{lem:full-frame}
Fix a packing of $Q_0,\ldots,Q_N$ of side $L<\Sigma$ and an index
$m\le N$, and suppose every triple of indices in $\{0,\ldots,m\}$ that
contains $0$ clears $\Sigma$. Then $Q_0$ can be translated to a corner in the
packing of $Q_0,\ldots,Q_m$ obtained by removing $Q_{m+1},\ldots,Q_N$. After reflection,
let $Q_0=[0,s_0]^2$ and put
\[
\mathcal R=\{i\in\{1,\ldots,m\}:x_i\ge s_0\},\qquad
\mathcal U=\{i\in\{1,\ldots,m\}:y_i\ge s_0\}.
\]
The arms cover $\{1,\ldots,m\}$. The right arm is a vertical chain with
pairwise x-overlaps, the upper arm is a horizontal chain with
pairwise y-overlaps, and $|\mathcal R\cap\mathcal U|\le1$.
If $Q_0$ is already cornered in the packing of $Q_0,\ldots,Q_N$, these
conclusions require no motion.
\end{lemma}
The proof of Proposition~\ref{prop:ell1} needs a refinement
of~\cite[Lemma~5.3]{BadgerSeven} in which the common index of the two
covering chains differs from $m$.
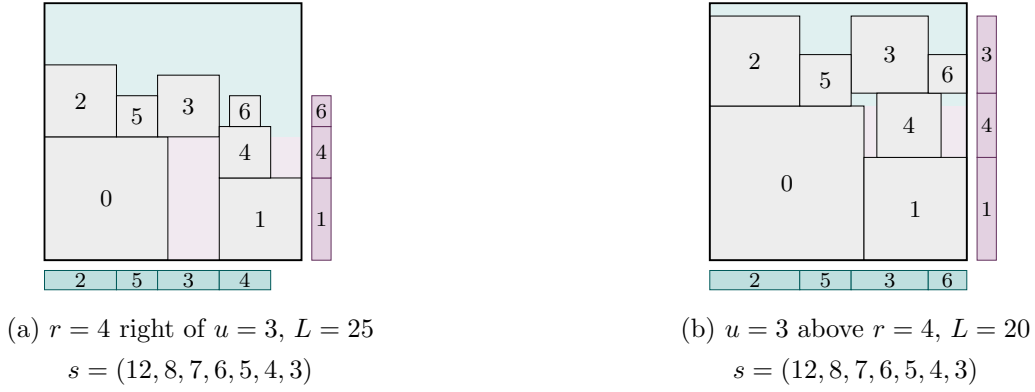
\begin{figure}[!htbp]\centering
\begin{minipage}[t]{0.45\textwidth}\centering
\begin{tikzpicture}[x=0.136cm,y=0.136cm]
\useasboundingbox (0,-3.375) rectangle (28.375,25);
\fill[vc!12] (12,0) rectangle (25,25);
\fill[hc!12] (0,12) rectangle (25,25);
\filldraw[fill=black!7,draw=black,line width=.35pt] (0,0) rectangle (12,12);
\node[font=\scriptsize,inner sep=0pt] at (6,6) {$0$};
\filldraw[fill=black!7,draw=black,line width=.35pt] (17,0) rectangle (25,8);
\node[font=\scriptsize,inner sep=0pt] at (21,4) {$1$};
\filldraw[fill=black!7,draw=black,line width=.35pt] (0,12) rectangle (7,19);
\node[font=\scriptsize,inner sep=0pt] at (3.5,15.5) {$2$};
\filldraw[fill=black!7,draw=black,line width=.35pt] (11,12) rectangle (17,18);
\node[font=\scriptsize,inner sep=0pt] at (14,15) {$3$};
\filldraw[fill=black!7,draw=black,line width=.35pt] (17,8) rectangle (22,13);
\node[font=\scriptsize,inner sep=0pt] at (19.5,10.5) {$4$};
\filldraw[fill=black!7,draw=black,line width=.35pt] (7,12) rectangle (11,16);
\node[font=\scriptsize,inner sep=0pt] at (9,14) {$5$};
\filldraw[fill=black!7,draw=black,line width=.35pt] (18,13) rectangle (21,16);
\node[font=\scriptsize,inner sep=0pt] at (19.5,14.5) {$6$};
\draw[line width=.7pt] (0,0) rectangle (25,25);
\filldraw[fill=hc!25,draw=hc!70!black,line width=.25pt] (0,-2.875) rectangle (7,-1);
\node[font=\tiny,inner sep=0pt] at (3.5,-1.9375) {$2$};
\filldraw[fill=hc!25,draw=hc!70!black,line width=.25pt] (7,-2.875) rectangle (11,-1);
\node[font=\tiny,inner sep=0pt] at (9,-1.9375) {$5$};
\filldraw[fill=hc!25,draw=hc!70!black,line width=.25pt] (11,-2.875) rectangle (17,-1);
\node[font=\tiny,inner sep=0pt] at (14,-1.9375) {$3$};
\filldraw[fill=hc!25,draw=hc!70!black,line width=.25pt] (17,-2.875) rectangle (22,-1);
\node[font=\tiny,inner sep=0pt] at (19.5,-1.9375) {$4$};
\filldraw[fill=vc!25,draw=vc!70!black,line width=.25pt] (26,0) rectangle (27.875,8);
\node[font=\tiny,inner sep=0pt] at (26.9375,4) {$1$};
\filldraw[fill=vc!25,draw=vc!70!black,line width=.25pt] (26,8) rectangle (27.875,13);
\node[font=\tiny,inner sep=0pt] at (26.9375,10.5) {$4$};
\filldraw[fill=vc!25,draw=vc!70!black,line width=.25pt] (26,13) rectangle (27.875,16);
\node[font=\tiny,inner sep=0pt] at (26.9375,14.5) {$6$};
\end{tikzpicture}\par
\smallskip{\footnotesize (a) $r=4$ right of $u=3$, $L=25$\\ $s=(12,8,7,6,5,4,3)$}
\end{minipage}
\hfill
\begin{minipage}[t]{0.45\textwidth}\centering
\begin{tikzpicture}[x=0.17cm,y=0.17cm]
\useasboundingbox (0,-2.7) rectangle (22.7,20);
\fill[vc!12] (12,0) rectangle (20,20);
\fill[hc!12] (0,12) rectangle (20,20);
\filldraw[fill=black!7,draw=black,line width=.35pt] (0,0) rectangle (12,12);
\node[font=\scriptsize,inner sep=0pt] at (6,6) {$0$};
\filldraw[fill=black!7,draw=black,line width=.35pt] (12,0) rectangle (20,8);
\node[font=\scriptsize,inner sep=0pt] at (16,4) {$1$};
\filldraw[fill=black!7,draw=black,line width=.35pt] (0,12) rectangle (7,19);
\node[font=\scriptsize,inner sep=0pt] at (3.5,15.5) {$2$};
\filldraw[fill=black!7,draw=black,line width=.35pt] (11,13) rectangle (17,19);
\node[font=\scriptsize,inner sep=0pt] at (14,16) {$3$};
\filldraw[fill=black!7,draw=black,line width=.35pt] (13,8) rectangle (18,13);
\node[font=\scriptsize,inner sep=0pt] at (15.5,10.5) {$4$};
\filldraw[fill=black!7,draw=black,line width=.35pt] (7,12) rectangle (11,16);
\node[font=\scriptsize,inner sep=0pt] at (9,14) {$5$};
\filldraw[fill=black!7,draw=black,line width=.35pt] (17,13) rectangle (20,16);
\node[font=\scriptsize,inner sep=0pt] at (18.5,14.5) {$6$};
\draw[line width=.7pt] (0,0) rectangle (20,20);
\filldraw[fill=hc!25,draw=hc!70!black,line width=.25pt] (0,-2.3) rectangle (7,-0.8);
\node[font=\tiny,inner sep=0pt] at (3.5,-1.55) {$2$};
\filldraw[fill=hc!25,draw=hc!70!black,line width=.25pt] (7,-2.3) rectangle (11,-0.8);
\node[font=\tiny,inner sep=0pt] at (9,-1.55) {$5$};
\filldraw[fill=hc!25,draw=hc!70!black,line width=.25pt] (11,-2.3) rectangle (17,-0.8);
\node[font=\tiny,inner sep=0pt] at (14,-1.55) {$3$};
\filldraw[fill=hc!25,draw=hc!70!black,line width=.25pt] (17,-2.3) rectangle (20,-0.8);
\node[font=\tiny,inner sep=0pt] at (18.5,-1.55) {$6$};
\filldraw[fill=vc!25,draw=vc!70!black,line width=.25pt] (20.8,0) rectangle (22.3,8);
\node[font=\tiny,inner sep=0pt] at (21.55,4) {$1$};
\filldraw[fill=vc!25,draw=vc!70!black,line width=.25pt] (20.8,8) rectangle (22.3,13);
\node[font=\tiny,inner sep=0pt] at (21.55,10.5) {$4$};
\filldraw[fill=vc!25,draw=vc!70!black,line width=.25pt] (20.8,13) rectangle (22.3,19);
\node[font=\tiny,inner sep=0pt] at (21.55,16) {$3$};
\end{tikzpicture}\par
\smallskip{\footnotesize (b) $u=3$ above $r=4$, $L=20$\\ $s=(12,8,7,6,5,4,3)$}
\end{minipage}
\caption{Illustrations of the arm geometry used in the proof of
Lemma~\ref{lem:two-cover}, with the arms of Lemma~\ref{lem:full-frame}
sharing $Q_6$; swapping axes interchanges the two orientations. The side
vectors do not satisfy the clearing hypotheses of
Lemma~\ref{lem:full-frame}. The bars record the covering chains:
$\mathcal R$ and $(\mathcal U\setminus\{6\})\cup\{4\}$ in (a), and
$(\mathcal R\setminus\{6\})\cup\{3\}$ and $\mathcal U$ in (b).}
\label{fig:cover8}
\end{figure}
\begin{lemma}[Perpendicular covering chains]\label{lem:two-cover}
In the frame of Lemma~\ref{lem:full-frame}, with $m\ge2$, there are a
vertical chain $\mathcal V$ and a horizontal chain $\mathcal H$ with
$\mathcal V\cup\mathcal H=\{1,\ldots,m\}$ and
$\mathcal V\cap\mathcal H=\{j\}$ for some $j\ne m$. In particular, for $m=6$ some four squares form a chain.
\end{lemma}
\begin{proof}
If the arms meet at $q\ne m$, they suffice. If either exclusive arm
$\mathcal R\setminus\mathcal U$ or $\mathcal U\setminus\mathcal R$ is empty,
one arm contains every index; pair it with a singleton whose index differs from $m$.

Otherwise let $r$ be topmost in $\mathcal R\setminus\mathcal U$ and $u$
rightmost in $\mathcal U\setminus\mathcal R$. Since $x_r\ge s_0>x_u$ and $y_u\ge s_0>y_r$, either $r$ lies right of $u$ or $u$ lies above $r$. Interchanging the axes
and the roles of $r$ and $u$ if necessary, assume that $r$ lies right of $u$;
Figure~\ref{fig:cover8} illustrates the two orientations.
Then $\mathcal R$ and $(\mathcal U\setminus\mathcal R)\cup\{r\}$
are covering chains with common index $r$. If $r\ne m$, they suffice.

Suppose $r=m$. The arms are disjoint: a common index other than $m$ was
already handled, while $m\in\mathcal R\setminus\mathcal U$.
If $\mathcal R=\{m\}$, then $\mathcal U\cup\{m\}$ is a full horizontal chain;
pair it with a singleton whose index differs from $m$. Otherwise let
$r'$ be topmost in $\mathcal R\setminus\{m\}$. If $r'$ is right of $u$,
use $\mathcal R$ and $\mathcal U\cup\{r'\}$. If not, $u$ lies above $r'$;
use $(\mathcal R\setminus\{m\})\cup\{u\}$ and
$\mathcal U\cup\{m\}$. In either case the common index differs from $m$.
For $m=6$ the two chain cardinalities sum to seven. Hence one is at least four.
\end{proof}

\section{Twenty-three strategies}\label{sec:eight}
Let $\alpha,\beta,\ldots,\psi$ be the twenty-three maxima of side sums in
Table~\ref{tab:pieces}, regarded as functions of $s$. Throughout, $\omega$
ranges over these twenty-three functions, and $\min_\omega$ denotes the minimum
over them. Set $F_8=\max(L_7^*,\min_\omega\omega)$.
The figures and Table~\ref{tab:coords8} use the same Greek letters for the
corresponding strategies.

\begin{figure}[p]
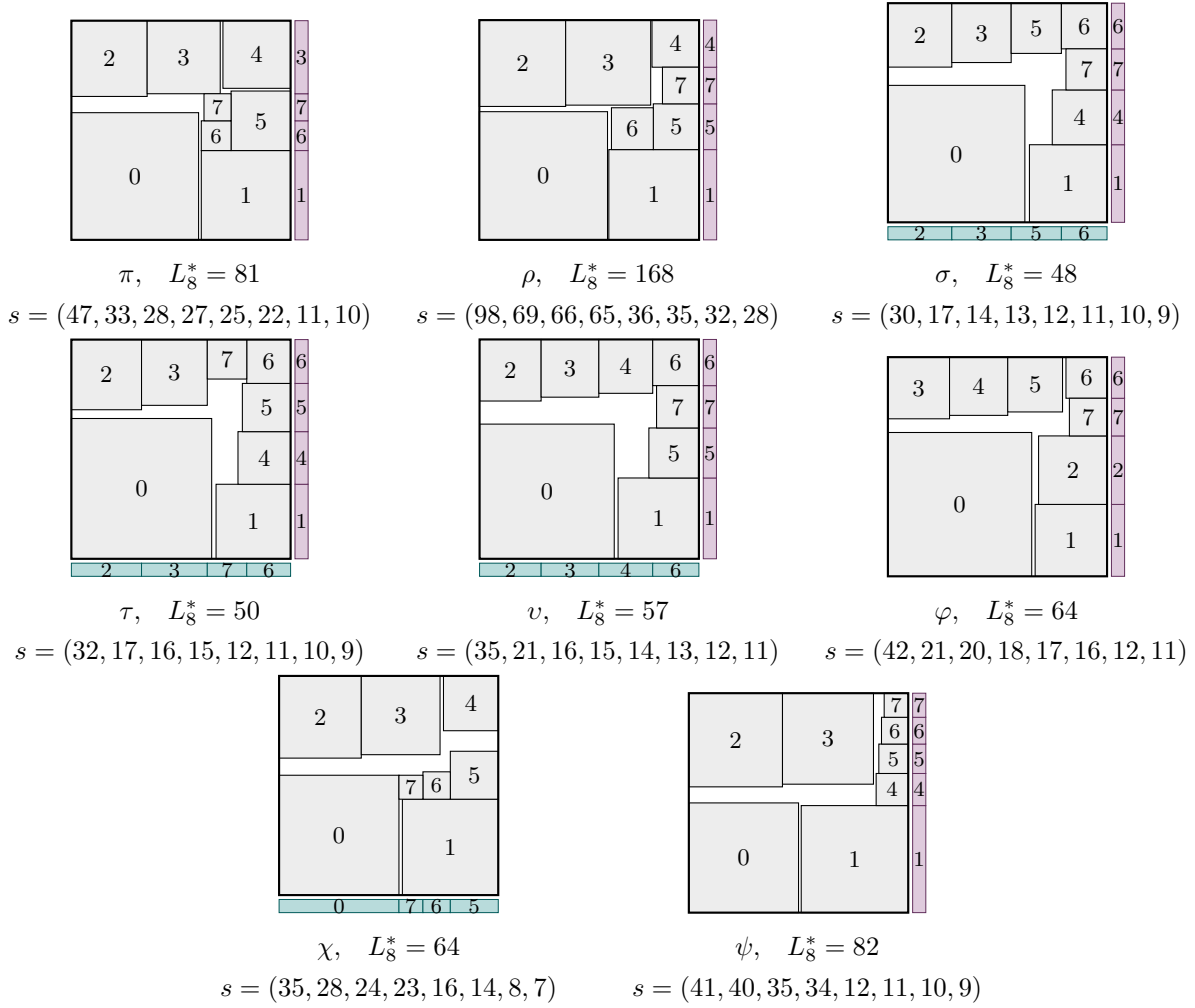

\centering
\noindent \begin{minipage}[t]{0.325\textwidth}\centering

\par\smallskip{\footnotesize $\psi$,\quad $L_8^*=82$\\ $s=(41,40,35,34,12,11,10,9)$}\par
\end{minipage}

\caption{Placements for the twenty-three strategies, $\alpha$ through
$\omicron$ on the previous page and $\pi$ through $\psi$ here, each at a side
vector with $L_8^*=L_7^*+1$ at which its maximum in Table~\ref{tab:pieces} is
the only one of the twenty-three equal to $L_8^*$. Square $Q_i$ carries
label $i$. A teal bar below a panel records the horizontal projections of a
chain with side sum $L_8^*$, and a purple bar at the right records a vertical
chain in the same way. The bars show every side sum in the maximum that
equals $L_8^*$. Panels are not to a common scale.}
\label{fig:strategies}
\end{figure}

\begin{proposition}[The twenty-three constructions]\label{prop:eight-upper}
For $L\ge\ell_\omega(s)=\max(L_6^*(s),\omega(s))$, the coordinates in the
row of $\omega$ in Table~\ref{tab:coords8} give a packing.
Their uniformly valid separation relations define a strategy
$\sigma_\omega$ with cost $\ell_\omega$. Consequently,
\[
L_8^*\le\min_\omega\ell_\omega
=\max\left(L_7^*,\min_\omega\omega\right).
\]
\end{proposition}
\begin{proof}
Fix $\omega$ and put $g(s)=\max(L_6^*(s),\omega(s))$.
Substitution in~\eqref{eq:contain} and~\eqref{eq:sep} verifies containment
and one fixed separation for each pair, uniformly for every
$s\in\ell_2^\downarrow$ and $L\ge g(s)$. Each check has the form
$aL+\sum_{i=0}^7c_i s_i\ge0$ with $a\ge0$. Every linear form in
$s_0,\ldots,s_7$ satisfies
\begin{equation}\label{eq:prefix-test}
\sum_{i=0}^7c_i s_i
=\sum_{k=0}^7\left(\sum_{i=0}^k c_i\right)(s_k-s_{k+1}),
\end{equation}
where $s_8=0$ for the truncated tuple. Hence a linear form with nonnegative
coefficient prefix sums is nonnegative. For each check, choose a side sum $S$
occurring in $g$ for which the linear form obtained by setting $L=S$ has
nonnegative coefficient prefix sums; the choice is made separately for each
containment or separation inequality. Since $L\ge g(s)\ge S$ and $a\ge0$,
this proves the required inequality. For example, in the row of $\psi$,
$Q_7$ lies above $Q_6$ because
$L-s_7\ge s_1+s_4+s_5+s_6$, and $Q_3$ lies left of $Q_4$ because
$s_2+s_3\le L-s_4$. These follow from $L\ge\psi(s)$ and
$L\ge s_2+s_3+s_4$, respectively.

The horizontal and vertical relations consisting of all uniformly valid
separations define a strategy $\sigma_\omega$ with
$\ell_{\sigma_\omega}\le g$ by~\cite[Lemma~3.3]{BadgerSeven}.
Each side sum in $\omega$ is the side sum of a chain of $\sigma_\omega$,
so $\ell_{\sigma_\omega}\ge\omega$. Lemma~\ref{lem:strategy} also gives
$\ell_{\sigma_\omega}\ge L_8^*\ge L_6^*$. Hence $\ell_{\sigma_\omega}=g$.
The code in~\cite{BadgerCode} verifies all the containment, separation,
and chain checks.

The constructions prove
\[
L_8^*\le\min_\omega\max(L_6^*,\omega)=\max\Bigl(L_6^*,\min_\omega\omega\Bigr).
\]
Deletion gives $L_6^*\le L_7^*\le L_8^*$. Since the common value of the two
expressions on the right is at least $L_8^*\ge L_7^*$, replacing $L_6^*$ by
$L_7^*$ in the last maximum does not change it. This gives the remaining
identities without using the eight-square lower bound.
\end{proof}

Figure~\ref{fig:strategies} shows each construction at a side
vector with $L_8^*>L_7^*$ at which it alone attains $L_8^*$. The values of
$L_8^*$ there follow from Theorem~\ref{thm:main}.

Deleting $Q_7$ from a strategy $\sigma$ for eight squares gives a strategy
$\sigma^-$ for the seven largest squares. The chains of $\sigma^-$ are the
chains of $\sigma$ that do not contain $7$, and removing $7$ from a chain of
$\sigma$ leaves a chain of $\sigma^-$ or the empty set. Hence
$\ell_{\sigma^-}(s_0,\ldots,s_6)=\ell_\sigma(s_0,\ldots,s_6,0)$, and
Proposition~\ref{prop:eight-upper} gives
\begin{equation}\label{eq:restrict}
\ell_{\sigma_\omega^-}(s)=\max\bigl(L_6^*(s),\ \omega(s_0,\ldots,s_6,0)\bigr).
\end{equation}
Comparing~\eqref{eq:restrict} with the three seven-square costs
in~\eqref{eq:intro-seven} shows that $\ell_{\sigma_\omega^-}$ is the cost of
$\st{center}$ for $\omega=\beta,\theta,\pi,\rho,\chi$, the cost of
$\st{top}$ for $\omega=\omicron,\varphi$, and the cost of $\st{right}$ for
$\omega=\nu,\tau,\psi$. For each of the other
thirteen strategies, $\ell_{\sigma_\omega^-}$ differs from all three
seven-square costs, and at the side vector illustrating $\sigma_\omega$ in
Figure~\ref{fig:strategies} it equals $L_7^*+1=L_8^*$. At the side
vector illustrating $\sigma_\tau$, its restriction also has cost
$L_7^*+1=L_8^*$, although its cost function is that of $\st{right}$. Thus,
fourteen of the twenty-three strategies have restrictions that cost more
than $L_7^*$ at the side vectors illustrating them.

\begingroup
\footnotesize
\renewcommand{\arraystretch}{1.0}

\begin{longtable}{c@{\qquad}llll}
\caption{Coordinates of the lower left corners of $Q_0,\dots,Q_7$ for the twenty-three strategies, as functions of the tuple $s$ and the side $L$.}\label{tab:coords8}\\
\toprule
 & $(x_0,y_0)$ & $(x_1,y_1)$ & $(x_2,y_2)$ & $(x_3,y_3)$\\
strategy & $(x_4,y_4)$ & $(x_5,y_5)$ & $(x_6,y_6)$ & $(x_7,y_7)$\\
\midrule
\endfirsthead
\multicolumn{5}{c}{\tablename~\thetable\ (continued)}\\[2pt]
\toprule
 & $(x_0,y_0)$ & $(x_1,y_1)$ & $(x_2,y_2)$ & $(x_3,y_3)$\\
strategy & $(x_4,y_4)$ & $(x_5,y_5)$ & $(x_6,y_6)$ & $(x_7,y_7)$\\
\midrule
\endhead
\bottomrule
\endlastfoot
$\alpha$ & $(0,\,0)$ & $(L-s_1,\,s_5)$ & $(0,\,L-s_2)$ & $(L-s_3,\,L-s_3)$ \\*
 & $(s_2,\,L-s_4)$ & $(L-s_5,\,0)$ & $(s_0,\,0)$ & $(s_2,\,s_0)$ \\[1.2pt]
$\beta$ & $(0,\,0)$ & $(L-s_1,\,0)$ & $(0,\,L-s_2)$ & $(s_2,\,L-s_3)$ \\*
 & $(L-s_4,\,L-s_4)$ & $(L-s_5,\,s_1)$ & $(L-s_5-s_6,\,s_1)$ & $(L-s_5-s_6-s_7,\,s_0)$ \\[1.2pt]
$\gamma$ & $(0,\,0)$ & $(L-s_1,\,0)$ & $(s_3,\,L-s_2)$ & $(0,\,L-s_3)$ \\*
 & $(L-s_4,\,L-s_4)$ & $(L-s_5,\,s_1)$ & $(L-s_5-s_6,\,s_1)$ & $(0,\,s_0)$ \\[1.2pt]
$\delta$ & $(0,\,0)$ & $(L-s_1,\,0)$ & $(s_4,\,L-s_2)$ & $(L-s_3,\,L-s_3)$ \\*
 & $(0,\,L-s_4)$ & $(0,\,s_0)$ & $(L-s_6,\,s_1)$ & $(L-s_6-s_7,\,s_1)$ \\[1.2pt]
$\varepsilon$ & $(0,\,0)$ & $(L-s_1,\,0)$ & $(0,\,L-s_2)$ & $(L-s_3,\,s_1)$ \\*
 & $(L-s_4,\,L-s_4)$ & $(s_2,\,s_0)$ & $(s_2+s_7,\,L-s_6)$ & $(s_2,\,L-s_7)$ \\[1.2pt]
$\zeta$ & $(0,\,0)$ & $(L-s_1,\,0)$ & $(L-s_2,\,s_1)$ & $(s_6,\,L-s_3)$ \\*
 & $(L-s_4,\,L-s_4)$ & $(s_3+s_6,\,L-s_5)$ & $(0,\,L-s_6)$ & $(0,\,s_0)$ \\[1.2pt]
$\eta$ & $(0,\,0)$ & $(L-s_1,\,0)$ & $(s_3,\,L-s_2)$ & $(0,\,L-s_3)$ \\*
 & $(L-s_4,\,L-s_4)$ & $(L-s_5,\,s_1)$ & $(0,\,s_0)$ & $(L-s_5-s_7,\,s_1)$ \\[1.2pt]
$\theta$ & $(0,\,0)$ & $(L-s_1,\,0)$ & $(0,\,L-s_2)$ & $(s_2,\,L-s_3)$ \\*
 & $(L-s_4,\,L-s_4)$ & $(L-s_5,\,s_1)$ & $(L-s_5-s_6,\,s_1)$ & $(0,\,s_0)$ \\[1.2pt]
$\iota$ & $(0,\,0)$ & $(L-s_1,\,0)$ & $(0,\,L-s_2)$ & $(L-s_3,\,L-s_3)$ \\*
 & $(s_2,\,L-s_4)$ & $(L-s_5,\,s_1)$ & $(L-s_5-s_6,\,s_1)$ & $(s_2,\,s_0)$ \\[1.2pt]
$\kappa$ & $(0,\,0)$ & $(L-s_1,\,0)$ & $(L-s_2,\,s_1)$ & $(0,\,L-s_3)$ \\*
 & $(L-s_4,\,L-s_4)$ & $(s_3,\,s_0)$ & $(s_3,\,L-s_6)$ & $(s_3+s_6,\,L-s_7)$ \\[1.2pt]
$\lambda$ & $(0,\,0)$ & $(L-s_1,\,0)$ & $(0,\,L-s_2)$ & $(L-s_3,\,L-s_3)$ \\*
 & $(s_2,\,L-s_4)$ & $(s_2,\,s_0)$ & $(L-s_6,\,s_1)$ & $(L-s_6-s_7,\,s_1)$ \\[1.2pt]
$\mu$ & $(0,\,0)$ & $(L-s_1,\,0)$ & $(0,\,L-s_2)$ & $(s_2,\,L-s_3)$ \\*
 & $(L-s_4,\,L-s_4)$ & $(L-s_5-s_6,\,s_1)$ & $(L-s_6,\,s_1)$ & $(L-s_7,\,s_1+s_6)$ \\[1.2pt]
$\nu$ & $(0,\,0)$ & $(L-s_1,\,0)$ & $(0,\,L-s_2)$ & $(s_2,\,L-s_3)$ \\*
 & $(L-s_4,\,L-s_4)$ & $(L-s_5,\,s_1+s_6)$ & $(L-s_6,\,s_1)$ & $(L-s_6-s_7,\,s_1)$ \\[1.2pt]
$\xi$ & $(0,\,0)$ & $(L-s_1,\,0)$ & $(0,\,L-s_2)$ & $(L-s_3,\,s_1)$ \\*
 & $(s_2,\,L-s_4)$ & $(s_2+s_4,\,L-s_5)$ & $(L-s_6,\,L-s_6)$ & $(L-s_7,\,s_1+s_3)$ \\[1.2pt]
$\omicron$ & $(0,\,0)$ & $(L-s_1,\,0)$ & $(L-s_2,\,s_1)$ & $(0,\,L-s_3)$ \\*
 & $(s_3,\,L-s_4)$ & $(L-s_5,\,L-s_5)$ & $(s_3+s_4,\,L-s_6)$ & $(s_3,\,L-s_4-s_7)$ \\[1.2pt]
$\pi$ & $(0,\,0)$ & $(L-s_1,\,0)$ & $(0,\,L-s_2)$ & $(s_2,\,L-s_3)$ \\*
 & $(L-s_4,\,L-s_4)$ & $(L-s_5,\,s_1)$ & $(L-s_5-s_6,\,s_1)$ & $(L-s_5-s_7,\,s_1+s_6)$ \\[1.2pt]
$\rho$ & $(0,\,0)$ & $(L-s_1,\,0)$ & $(0,\,L-s_2)$ & $(s_2,\,L-s_3)$ \\*
 & $(L-s_4,\,L-s_4)$ & $(L-s_5,\,s_1)$ & $(L-s_5-s_6,\,s_1)$ & $(L-s_7,\,s_1+s_5)$ \\[1.2pt]
$\sigma$ & $(0,\,0)$ & $(L-s_1,\,0)$ & $(0,\,L-s_2)$ & $(s_2,\,L-s_3)$ \\*
 & $(L-s_4,\,s_1)$ & $(s_2+s_3,\,L-s_5)$ & $(L-s_6,\,L-s_6)$ & $(L-s_7,\,s_1+s_4)$ \\[1.2pt]
$\tau$ & $(0,\,0)$ & $(L-s_1,\,0)$ & $(0,\,L-s_2)$ & $(s_2,\,L-s_3)$ \\*
 & $(L-s_4,\,s_1)$ & $(L-s_5,\,s_1+s_4)$ & $(L-s_6,\,L-s_6)$ & $(s_2+s_3,\,L-s_7)$ \\[1.2pt]
$\upsilon$ & $(0,\,0)$ & $(L-s_1,\,0)$ & $(0,\,L-s_2)$ & $(s_2,\,L-s_3)$ \\*
 & $(s_2+s_3,\,L-s_4)$ & $(L-s_5,\,s_1)$ & $(L-s_6,\,L-s_6)$ & $(L-s_7,\,s_1+s_5)$ \\[1.2pt]
$\varphi$ & $(0,\,0)$ & $(L-s_1,\,0)$ & $(L-s_2,\,s_1)$ & $(0,\,L-s_3)$ \\*
 & $(s_3,\,L-s_4)$ & $(s_3+s_4,\,L-s_5)$ & $(L-s_6,\,L-s_6)$ & $(L-s_7,\,s_1+s_2)$ \\[1.2pt]
$\chi$ & $(0,\,0)$ & $(L-s_1,\,0)$ & $(0,\,L-s_2)$ & $(s_2,\,L-s_3)$ \\*
 & $(L-s_4,\,L-s_4)$ & $(L-s_5,\,s_1)$ & $(L-s_5-s_6,\,s_1)$ & $(L-s_5-s_6-s_7,\,s_1)$ \\[1.2pt]
$\psi$ & $(0,\,0)$ & $(L-s_1,\,0)$ & $(0,\,L-s_2)$ & $(s_2,\,L-s_3)$ \\*
 & $(L-s_4,\,s_1)$ & $(L-s_5,\,s_1+s_4)$ & $(L-s_6,\,s_1+s_4+s_5)$ & $(L-s_7,\,L-s_7)$ \\[1.2pt]
\end{longtable}
\endgroup

\section{Lower bounds}\label{sec:lower}
We prove eight lower bounds $L_8^*(s)\ge b_i(s)$, $i=1,\ldots,8$, each
given by side sums of $s_0,\ldots,s_7$, in
Propositions~\ref{prop:ell1}--\ref{prop:ell9}. Section~\ref{sec:combine}
combines them with $L_7^*$ and the column bound of
Lemma~\ref{lem:column}. Table~\ref{tab:names}
records the notation for recurring side sums.

The bounds $b_2,\ldots,b_8$ are minima of side sums. In their proofs, we
set $\Sigma=b_i$ and suppose that a packing has side
$L<\Sigma$. A family of squares clears $\Sigma$ as soon as its side sum is at
least one of the side sums in $b_i$. Since $s$ is nonincreasing, replacing
indices in a family by larger ones does not increase its side sum. These
proofs follow a common pattern. In the
cornering arguments, some of the resulting exclusions allow $Q_0$ to be moved to a corner and restrict the positions of the next
largest squares. Further exclusions and comparisons of side sums then
rule out the remaining placements, ending with a chain that clears
$\Sigma$. The proof of $b_3$ needs no cornering; it begins with the
weighted-chain comparison of Lemma~\ref{lem:weighted}. The bound $b_1$
instead combines two perpendicular covering chains with a comparison of
partitions.

\begingroup
\small
\renewcommand{\arraystretch}{1.12}
\begin{longtable}{@{}l@{\qquad}l@{\qquad}l@{}}
\caption{Recurring side sums.}\label{tab:names}\\
\toprule
\endfirsthead
\multicolumn{3}{c}{\tablename~\thetable\ (continued)}\\[2pt]
\toprule
\endhead
\bottomrule
\endlastfoot
$W=s_3+s_4+s_5+s_6$ & $G_6=s_0+s_3+s_6$ & $J_7=s_1+s_2+s_7$\\
$H=s_1+s_2+s_5$ & $G_7=s_0+s_3+s_7$ & $J_{67}=s_1+s_2+s_6+s_7$\\
$C_0=s_0+s_5+s_6$ & $K=s_0+s_2+s_7$ & $M=s_1+s_4+s_6+s_7$\\
$C_3=s_1+s_3+s_6$ & $T=s_0+s_6+s_7$ & $\Ppi=s_1+s_3+s_6+s_7$\\
$R=s_1+s_4+s_5+s_6$ & $U=s_0+s_4+s_5$ & $\Prho=s_1+s_4+s_5+s_7$\\
$D=s_1+s_3+s_7$ & $H_4=s_1+s_3+s_4$ & $\Pchi=s_0+s_5+s_6+s_7$\\
$E=s_0+s_4+s_7$ & $H_5=s_1+s_3+s_5$ & $V=s_2+s_5+s_6+s_7$\\
$P=s_1+s_4+s_5+s_6+s_7$ & $J_4=s_1+s_2+s_4$ & $X=s_3+s_4+s_6+s_7$\\
 & $J_6=s_1+s_2+s_6$ & $Y=s_2+s_4+s_6+s_7$\\
\end{longtable}
\endgroup

\needspace{11\baselineskip}
\subsection{The balanced-partition bound}
\begin{proposition}\label{prop:ell1}
For every $s\in\ell_2^\downarrow$,
\[
L_8^*(s)\ge b_1(s):=\min\bigl(T,\ \mathcal K(s)\bigr),
\]
where $T=s_0+s_6+s_7$ and
\[
\mathcal K(s)=\min_{\substack{I\sqcup J=\{1,2,3,4,5,7\}\\|I|=|J|=3}}
\max\left(s_6+\sum_{i\in I}s_i,\ s_6+\sum_{j\in J}s_j\right).
\]
\end{proposition}
\begin{proof}
Suppose $L<\Sigma:=b_1$.
A triple containing $Q_0$ has side sum at least $T$, because its other two
squares are no smaller than $Q_6$ and $Q_7$. Thus, every such triple clears
$\Sigma$. Apply Lemmas~\ref{lem:full-frame} and
\ref{lem:two-cover} with $m=7$. They give perpendicular covering chains
with common index $j\le6$.

If $j\ne6$, interchange the labels $j$ and $6$ in the two index
sets. The family containing both has unchanged side sum, while the
other sum decreases by $s_j-s_6\ge0$. This is a statement about sums only;
the relabelled families need not remain geometric chains. Removing the
new common index $6$ gives a partition
\[
I\sqcup J=\{1,2,3,4,5,7\},\qquad
s_6+\sum_{i\in I}s_i\le L,\quad s_6+\sum_{i\in J}s_i\le L.
\]
Because $s_0+s_1\le L<T$ by Lemma~\ref{lem:chain}, we have $s_1<s_6+s_7\le s_5+s_7$.
Among the six numbers indexed by $1,2,3,4,5,7$, every pair therefore sums
to more than every singleton.

Under this pair--singleton property, any two-part partition can be
replaced by a $3+3$ partition without increasing its maximum part-sum.
For a $2+4$ partition, write the first two entries as $a,b$ and split the
other four into pairs of sums $u,v$. Then $a+u,b+v\le u+v$. For a $1+5$
partition, join the singleton to two entries of the other part. Since the three
remaining entries sum to at least the singleton, neither new triple
exceeds the old five-term sum. The $0+6$ case is immediate.
The resulting $3+3$ partition gives $\KK\le L$, contrary to $L<\Sigma$.
\end{proof}

\subsection{The bridge bound}\label{sec:bridge}
\begin{proposition}\label{prop:ell3}
For every $s\in\ell_2^\downarrow$,
\[
L_8^*(s)\ge b_2(s):=\min(C_0,\ E,\ J_4,\ J_{67}),
\]
where
\[
C_0=s_0+s_5+s_6,\quad E=s_0+s_4+s_7,\quad J_4=s_1+s_2+s_4,\quad
J_{67}=s_1+s_2+s_6+s_7.
\]
\end{proposition}
\begin{figure}[t]\centering
\begin{minipage}[t]{0.45\textwidth}\centering
\begin{tikzpicture}[x=0.0582cm,y=0.0582cm]
\useasboundingbox (-7.38,-9.43) rectangle (82,82);
\filldraw[fill=black!7,draw=black,line width=.35pt] (0,0) rectangle (53,53);
\node[font=\scriptsize,inner sep=0pt] at (26.5,26.5) {$0$};
\filldraw[fill=black!7,draw=black,line width=.35pt] (53,0) rectangle (82,29);
\node[font=\scriptsize,inner sep=0pt] at (67.5,14.5) {$1$};
\filldraw[fill=black!7,draw=black,line width=.35pt] (54,29) rectangle (82,57);
\node[font=\scriptsize,inner sep=0pt] at (68,43) {$2$};
\filldraw[fill=black!7,draw=black,line width=.35pt] (0,62) rectangle (20,82);
\node[font=\scriptsize,inner sep=0pt] at (10,72) {$3$};
\filldraw[fill=black!7,draw=black,line width=.35pt] (20,63) rectangle (39,82);
\node[font=\scriptsize,inner sep=0pt] at (29.5,72.5) {$4$};
\filldraw[fill=black!7,draw=black,line width=.35pt] (39,65) rectangle (56,82);
\node[font=\scriptsize,inner sep=0pt] at (47.5,73.5) {$5$};
\filldraw[fill=black!7,draw=black,line width=.35pt] (67,67) rectangle (82,82);
\node[font=\scriptsize,inner sep=0pt] at (74.5,74.5) {$6$};
\filldraw[fill=black!7,draw=black,line width=.35pt] (56,71) rectangle (67,82);
\node[font=\scriptsize,inner sep=0pt] at (61.5,76.5) {$7$};
\draw[densely dashed,line width=.3pt] (53,0) -- (53,82);
\draw[densely dashed,line width=.3pt] (0,53) -- (82,53);
\node[font=\tiny,inner sep=1pt,left] at (0,53) {$s_0$};
\draw[line width=.7pt] (0,0) rectangle (82,82);
\filldraw[fill=hc!25,draw=hc!70!black,line width=.25pt] (0,-8.2) rectangle (20,-3.28);
\node[font=\tiny,inner sep=0pt] at (10,-5.74) {$3$};
\filldraw[fill=hc!25,draw=hc!70!black,line width=.25pt] (20,-8.2) rectangle (39,-3.28);
\node[font=\tiny,inner sep=0pt] at (29.5,-5.74) {$4$};
\filldraw[fill=hc!25,draw=hc!70!black,line width=.25pt] (39,-8.2) rectangle (56,-3.28);
\node[font=\tiny,inner sep=0pt] at (47.5,-5.74) {$5$};
\filldraw[fill=hc!25,draw=hc!70!black,line width=.25pt] (56,-8.2) rectangle (67,-3.28);
\node[font=\tiny,inner sep=0pt] at (61.5,-5.74) {$7$};
\filldraw[fill=hc!25,draw=hc!70!black,line width=.25pt] (67,-8.2) rectangle (82,-3.28);
\node[font=\tiny,inner sep=0pt] at (74.5,-5.74) {$6$};
\end{tikzpicture}\par
\smallskip{\footnotesize (a) $Q_7$ right of $Q_0$, $L=82$\\ $s=(53,29,28,20,19,17,15,11)$}
\end{minipage}
\hfill
\begin{minipage}[t]{0.45\textwidth}\centering
\begin{tikzpicture}[x=0.0691cm,y=0.0691cm]
\useasboundingbox (-6.21,-7.935) rectangle (69,69);
\filldraw[fill=black!7,draw=black,line width=.35pt] (0,0) rectangle (44,44);
\node[font=\scriptsize,inner sep=0pt] at (22,22) {$0$};
\filldraw[fill=black!7,draw=black,line width=.35pt] (44,0) rectangle (69,25);
\node[font=\scriptsize,inner sep=0pt] at (56.5,12.5) {$1$};
\filldraw[fill=black!7,draw=black,line width=.35pt] (46,25) rectangle (69,48);
\node[font=\scriptsize,inner sep=0pt] at (57.5,36.5) {$2$};
\filldraw[fill=black!7,draw=black,line width=.35pt] (0,52) rectangle (17,69);
\node[font=\scriptsize,inner sep=0pt] at (8.5,60.5) {$3$};
\filldraw[fill=black!7,draw=black,line width=.35pt] (17,54) rectangle (32,69);
\node[font=\scriptsize,inner sep=0pt] at (24.5,61.5) {$4$};
\filldraw[fill=black!7,draw=black,line width=.35pt] (43,55) rectangle (57,69);
\node[font=\scriptsize,inner sep=0pt] at (50,62) {$5$};
\filldraw[fill=black!7,draw=black,line width=.35pt] (57,57) rectangle (69,69);
\node[font=\scriptsize,inner sep=0pt] at (63,63) {$6$};
\filldraw[fill=black!7,draw=black,line width=.35pt] (32,58) rectangle (43,69);
\node[font=\scriptsize,inner sep=0pt] at (37.5,63.5) {$7$};
\draw[densely dashed,line width=.3pt] (44,0) -- (44,69);
\draw[densely dashed,line width=.3pt] (0,44) -- (69,44);
\node[font=\tiny,inner sep=1pt,left] at (0,44) {$s_0$};
\draw[line width=.7pt] (0,0) rectangle (69,69);
\filldraw[fill=hc!25,draw=hc!70!black,line width=.25pt] (0,-6.9) rectangle (17,-2.76);
\node[font=\tiny,inner sep=0pt] at (8.5,-4.83) {$3$};
\filldraw[fill=hc!25,draw=hc!70!black,line width=.25pt] (17,-6.9) rectangle (32,-2.76);
\node[font=\tiny,inner sep=0pt] at (24.5,-4.83) {$4$};
\filldraw[fill=hc!25,draw=hc!70!black,line width=.25pt] (32,-6.9) rectangle (43,-2.76);
\node[font=\tiny,inner sep=0pt] at (37.5,-4.83) {$7$};
\filldraw[fill=hc!25,draw=hc!70!black,line width=.25pt] (46,-6.9) rectangle (69,-2.76);
\node[font=\tiny,inner sep=0pt] at (57.5,-4.83) {$2$};
\end{tikzpicture}\par
\smallskip{\footnotesize (b) $Q_7$ above $Q_0$, $L=69$\\ $s=(44,25,23,17,15,14,12,11)$}
\end{minipage}
\caption{The two placements of $Q_7$ in the proof of
Proposition~\ref{prop:ell3}; the dashed lines are $x=s_0$ and $y=s_0$. Here
$L$ is less than $C_0$, $E$, and $J_{67}$, which are $85$, $83$, $83$ in (a)
and $70$, $70$, $71$ in (b), but $J_4\le L$. Squares $Q_1,Q_2$ form the
exclusive right arm with upper square $t=2$, and $Q_3,Q_4$ lie left of
$Q_t$. Square $Q_6$ lies in the right arm, and the bridge $u=5$ satisfies
\eqref{eq:bridge-cross}. (a) $Q_7$ lies right of $Q_0$ and above $Q_t$; the
bar records the horizontal chain $\{3,4,5,6,7\}$. (b) $Q_7$ lies above $Q_0$
and is x-disjoint from $Q_t$; the bar records the horizontal chain
$3,4,7,t$. If the remaining inequality $L<J_4$ also held, each indicated
chain would have side sum greater than $L$ by~\eqref{eq:bridge-end}, a
contradiction.}
\label{fig:bridge}
\end{figure}
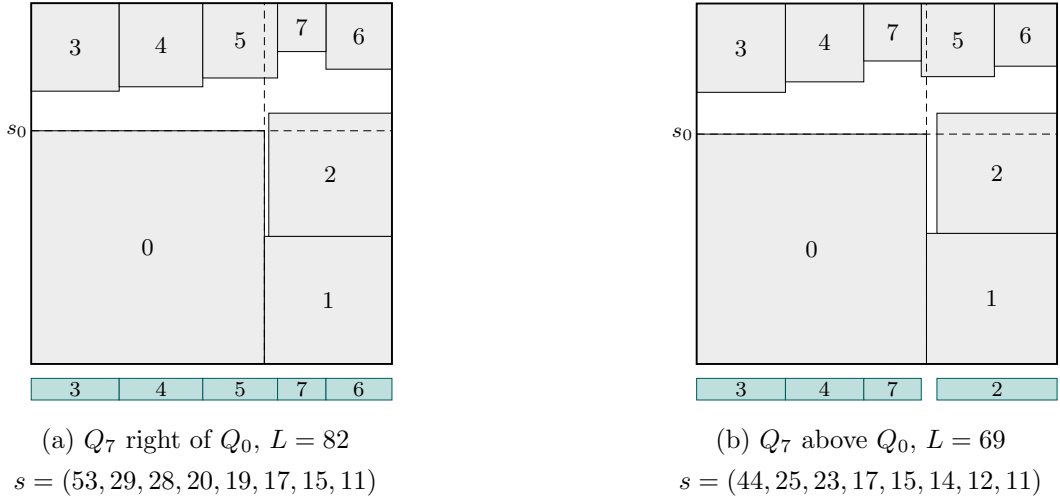
\begin{proof}
Fix $\Sigma=b_2$ and suppose $L<\Sigma$.
Since $L\ge s_0+s_1$, the inequalities $L<C_0$ and $L<E$ give
\begin{equation}\label{eq:bridge-arith}
s_5+s_6>s_1,\qquad s_4+s_7>s_1.
\end{equation}
Any three of $Q_0,\ldots,Q_3$ have side sum at least $s_1+s_2+s_3\ge J_4$.
A triple consisting of $Q_0$, one of $Q_1,Q_2,Q_3$, and any third square has
side sum at least $s_0+s_3+s_7\ge E$. All these triples clear $\Sigma$, and
Lemma~\ref{lem:pane} therefore corners $Q_0$ with all eight squares present.

We first show that $Q_1,\ldots,Q_6$ contain a four-square chain, and that
it must be $\{Q_3,Q_4,Q_5,Q_6\}$. Among $Q_0,\ldots,Q_6$, every triple
containing $Q_0$ has side sum at least $C_0$ and clears $\Sigma$. Since $Q_0$ is already cornered, the
no-motion clause of Lemma~\ref{lem:full-frame} applies to the prefix
$Q_0,\ldots,Q_6$. Every four-subset of $\{1,\ldots,6\}$ except
$\mathcal W=\{3,4,5,6\}$ has sum at least
$s_2+s_4+s_5+s_6>J_4\ge\Sigma$ by \eqref{eq:bridge-arith}. Lemma~\ref{lem:two-cover} forces a four-chain,
which must therefore be $\mathcal W$. Interchange axes so that it is
horizontal. Members of the right arm pairwise x-overlap, whereas members
of $\mathcal W$ are pairwise x-disjoint. Hence at most one member of
$\mathcal W$ is in the right arm, and at least three are in the upper arm.
If $1$ or $2$ were in the upper arm, it and
those three would give a clearing four-chain different from $\mathcal W$.
Thus, $1,2$ lie in the exclusive right arm, form a vertical chain, and have
lower edges below $s_0$. Let $t$ be their upper square.

The triples $1,2,3$ and $1,2,4$ have side sum at least $J_4$. They clear
$\Sigma$ and force $3,4$ into
the exclusive upper arm. Neither can x-overlap $t$: it
would lie above $t$ and give the corresponding clearing vertical triple
with $1,2$. Since $x_3,x_4<s_0\le x_t$, both lie left of $t$.

Each of $5,6$ must x-overlap $t$, or else it and
$t,3,4$ would be a clearing horizontal quadruple. The two squares
$5,6$ are x-disjoint from one another. Exactly one belongs to the right
arm. There is at most one by the arm-overlap property. If neither did,
both would start left of $s_0$ and end right of $x_t\ge s_0$, contradicting
their x-disjointness. Denote the other, exclusive upper-arm square by $u$.
Then
\begin{equation}\label{eq:bridge-cross}
x_u<s_0\le x_t<x_u+s_u,
\end{equation}
and $u$ lies above $t$. We call $Q_u$ the bridge: it crosses the line
$x=s_0$ and reaches past the left edge of $Q_t$. The right-arm member of $\{5,6\}$ also lies
above $t$. Otherwise it would lie below $t$, and since $u$ lies above $t$,
the four squares $1,2,5,6$ would be a vertical chain of sum at least
$J_{67}$. Both $5,6$ now lie above $t$, and the same sum shows that they
y-overlap.

It remains to place the eighth square.
Figure~\ref{fig:bridge} shows the two placements of $Q_7$ treated below.
Two further side sums exceed $J_4\ge\Sigma$. With $W=s_3+s_4+s_5+s_6$,
\begin{equation}\label{eq:bridge-end}
s_2+s_3+s_4+s_7>J_4,\qquad W+s_7>J_4.
\end{equation}
For the first, $s_3+s_7\ge s_4+s_7>s_1$. For the second, subtract $J_4$;
the difference is $(s_5+s_6-s_1)+(s_3+s_7-s_2)>0$ by
\eqref{eq:bridge-arith}, because $s$ is nonincreasing.

If $7$ is x-disjoint from $0$, it lies right of $0$. The horizontal pins
$0,1,7$ and $0,2,7$ make it x-overlap both $1,2$. Hence it is
y-disjoint from both. For $k=5,6$, the quadruple $1,2,k,7$ has side sum at least $J_{67}$, giving a
vertical pin of $7$ to the chain $1,2,k$. Thus, $7$ y-overlaps both $5,6$,
and is x-disjoint from both. In particular it lies right of $u$, since
$x_7\ge s_0>x_u$. By \eqref{eq:bridge-cross}, $x_7>x_t$. Hence it is also
x-disjoint from $3,4$, which lie left of $t$. The five squares
$3,4,5,6,7$ now form a horizontal chain that clears $\Sigma$
by~\eqref{eq:bridge-end}.

Otherwise $7$ x-overlaps $0$ and lies above it. The triples
$0,3,7$ and $0,4,7$ have side sum at least $E$, and their vertical pins make $7$
x-disjoint from $3,4$.
If it is also x-disjoint from $t$, the horizontal chain $3,4,7,t$ has side
sum at least $s_2+s_3+s_4+s_7$, which exceeds $J_4$ by \eqref{eq:bridge-end}. If not, it is above $t$, since
$y_7\ge s_0>y_t$, and hence y-disjoint from $1,2$. The same pins, from the quadruples $1,2,k,7$,
make it x-disjoint from $5,6$. Again $3,4,5,6,7$ form the forbidden
horizontal five-chain. This proves the bound.
\end{proof}

\subsection{The weighted-chain bound}\label{sec:weighted-bound}
\begin{proposition}\label{prop:ell4}
For every $s\in\ell_2^\downarrow$,
\[
L_8^*(s)\ge b_3(s):=\min(J_7,\ V),
\]
where $J_7=s_1+s_2+s_7$ and $V=s_2+s_5+s_6+s_7$.
\end{proposition}
\begin{proof}
Suppose $L<\Sigma:=b_3$. For this proof call $0,1,2$ large and
$3,4,5,6,7$ small. A
chain containing two large squares and a third square, or one large square
and three small squares, has side sum at least $J_7$ or $V$, respectively,
and clears $\Sigma$.

Give the large squares weight two and the small squares weight one. If
every chain had weight at most four, Lemma~\ref{lem:weighted} would give
$3\cdot2^2+5\cdot1^2=17\le4^2=16$, a contradiction.
Thus, some chain has weight at least five. If it contains a large square,
it contains one of the two forbidden types above. Hence all five small
squares form a chain. Interchange axes so that it is horizontal, and let
$q$ be its middle square in left-to-right order.

Every large square x-overlaps $q$. Otherwise it lies on one
side of $q$ and is x-disjoint from at least three small squares, giving
a clearing horizontal chain. Thus, $q$ is y-disjoint from each large square.
Two large squares cannot lie on opposite sides of $q$ vertically, since
they would form a clearing triple with it. Nor can two on the same side
be y-disjoint from one another, for the same reason. The three large
squares therefore y-overlap pairwise and form a clearing
horizontal triple. This contradiction proves the bound.
\end{proof}

\subsection{Two bounds from a right-hand stack}\label{sec:stack}
In this subsection $Q_0,\ldots,Q_3$ are large, $Q_4,Q_5$ are medium, and
$Q_6,Q_7$ are small. We first show that the medium squares $Q_4,Q_5$ form a
vertical stack on the right, above $Q_1$; this is the common part of the two
proofs.
Figure~\ref{fig:stack} shows the configurations that it produces.

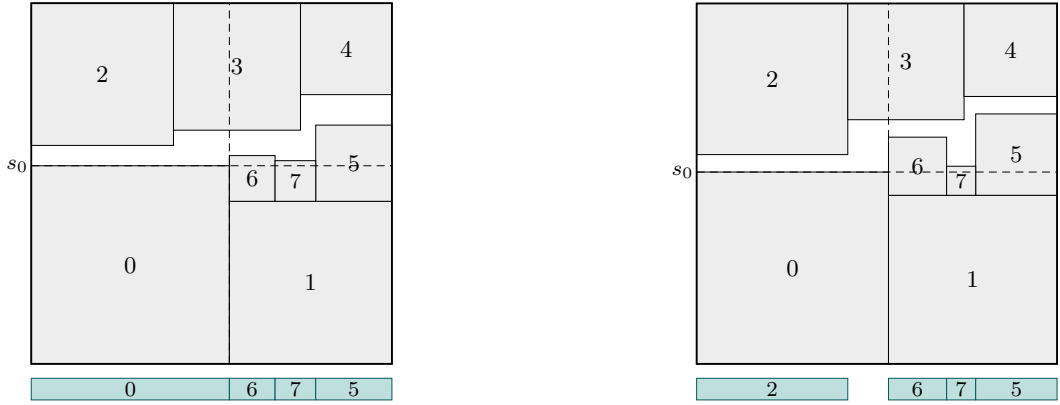
\begin{figure}[!htbp]\centering
\begin{minipage}[t]{0.45\textwidth}\centering
\begin{tikzpicture}[x=0.0672cm,y=0.0672cm]
\useasboundingbox (-6.39,-8.165) rectangle (71,71);
\filldraw[fill=black!7,draw=black,line width=.35pt] (0,0) rectangle (39,39);
\node[font=\scriptsize,inner sep=0pt] at (19.5,19.5) {$0$};
\filldraw[fill=black!7,draw=black,line width=.35pt] (39,0) rectangle (71,32);
\node[font=\scriptsize,inner sep=0pt] at (55,16) {$1$};
\filldraw[fill=black!7,draw=black,line width=.35pt] (0,43) rectangle (28,71);
\node[font=\scriptsize,inner sep=0pt] at (14,57) {$2$};
\filldraw[fill=black!7,draw=black,line width=.35pt] (28,46) rectangle (53,71);
\node[font=\scriptsize,inner sep=0pt] at (40.5,58.5) {$3$};
\filldraw[fill=black!7,draw=black,line width=.35pt] (53,53) rectangle (71,71);
\node[font=\scriptsize,inner sep=0pt] at (62,62) {$4$};
\filldraw[fill=black!7,draw=black,line width=.35pt] (56,32) rectangle (71,47);
\node[font=\scriptsize,inner sep=0pt] at (63.5,39.5) {$5$};
\filldraw[fill=black!7,draw=black,line width=.35pt] (39,32) rectangle (48,41);
\node[font=\scriptsize,inner sep=0pt] at (43.5,36.5) {$6$};
\filldraw[fill=black!7,draw=black,line width=.35pt] (48,32) rectangle (56,40);
\node[font=\tiny,inner sep=0pt] at (52,36) {$7$};
\draw[densely dashed,line width=.3pt] (39,0) -- (39,71);
\draw[densely dashed,line width=.3pt] (0,39) -- (71,39);
\node[font=\tiny,inner sep=1pt,left] at (0,39) {$s_0$};
\draw[line width=.7pt] (0,0) rectangle (71,71);
\filldraw[fill=hc!25,draw=hc!70!black,line width=.25pt] (0,-7.1) rectangle (39,-2.84);
\node[font=\tiny,inner sep=0pt] at (19.5,-4.97) {$0$};
\filldraw[fill=hc!25,draw=hc!70!black,line width=.25pt] (39,-7.1) rectangle (48,-2.84);
\node[font=\tiny,inner sep=0pt] at (43.5,-4.97) {$6$};
\filldraw[fill=hc!25,draw=hc!70!black,line width=.25pt] (48,-7.1) rectangle (56,-2.84);
\node[font=\tiny,inner sep=0pt] at (52,-4.97) {$7$};
\filldraw[fill=hc!25,draw=hc!70!black,line width=.25pt] (56,-7.1) rectangle (71,-2.84);
\node[font=\tiny,inner sep=0pt] at (63.5,-4.97) {$5$};
\end{tikzpicture}\par
\smallskip{\footnotesize (a) Proposition~\ref{prop:ell5}, $L=71$\\ $s=(39,32,28,25,18,15,9,8)$}
\end{minipage}
\hfill
\begin{minipage}[t]{0.45\textwidth}\centering
\begin{tikzpicture}[x=0.0769cm,y=0.0769cm]
\useasboundingbox (-5.58,-7.13) rectangle (62,62);
\filldraw[fill=black!7,draw=black,line width=.35pt] (0,0) rectangle (33,33);
\node[font=\scriptsize,inner sep=0pt] at (16.5,16.5) {$0$};
\filldraw[fill=black!7,draw=black,line width=.35pt] (33,0) rectangle (62,29);
\node[font=\scriptsize,inner sep=0pt] at (47.5,14.5) {$1$};
\filldraw[fill=black!7,draw=black,line width=.35pt] (0,36) rectangle (26,62);
\node[font=\scriptsize,inner sep=0pt] at (13,49) {$2$};
\filldraw[fill=black!7,draw=black,line width=.35pt] (26,42) rectangle (46,62);
\node[font=\scriptsize,inner sep=0pt] at (36,52) {$3$};
\filldraw[fill=black!7,draw=black,line width=.35pt] (46,46) rectangle (62,62);
\node[font=\scriptsize,inner sep=0pt] at (54,54) {$4$};
\filldraw[fill=black!7,draw=black,line width=.35pt] (48,29) rectangle (62,43);
\node[font=\scriptsize,inner sep=0pt] at (55,36) {$5$};
\filldraw[fill=black!7,draw=black,line width=.35pt] (33,29) rectangle (43,39);
\node[font=\scriptsize,inner sep=0pt] at (38,34) {$6$};
\filldraw[fill=black!7,draw=black,line width=.35pt] (43,29) rectangle (48,34);
\node[font=\tiny,inner sep=0pt] at (45.5,31.5) {$7$};
\draw[densely dashed,line width=.3pt] (33,0) -- (33,62);
\draw[densely dashed,line width=.3pt] (0,33) -- (62,33);
\node[font=\tiny,inner sep=1pt,left] at (0,33) {$s_0$};
\draw[line width=.7pt] (0,0) rectangle (62,62);
\filldraw[fill=hc!25,draw=hc!70!black,line width=.25pt] (0,-6.2) rectangle (26,-2.48);
\node[font=\tiny,inner sep=0pt] at (13,-4.34) {$2$};
\filldraw[fill=hc!25,draw=hc!70!black,line width=.25pt] (33,-6.2) rectangle (43,-2.48);
\node[font=\tiny,inner sep=0pt] at (38,-4.34) {$6$};
\filldraw[fill=hc!25,draw=hc!70!black,line width=.25pt] (43,-6.2) rectangle (48,-2.48);
\node[font=\tiny,inner sep=0pt] at (45.5,-4.34) {$7$};
\filldraw[fill=hc!25,draw=hc!70!black,line width=.25pt] (48,-6.2) rectangle (62,-2.48);
\node[font=\tiny,inner sep=0pt] at (55,-4.34) {$5$};
\end{tikzpicture}\par
\smallskip{\footnotesize (b) Proposition~\ref{prop:ell6}, $L=62$\\ $s=(33,29,26,20,16,14,10,5)$}
\end{minipage}
\caption{The right-hand stack of Lemma~\ref{lem:right-stack}, with
$p=2$, $c=3$, and medium squares $b=5$ (lower) and $t=4$ (upper); the dashed
lines are $x=s_0$ and $y=s_0$. In both panels every family in the hypotheses
of the lemma has side sum greater than $L$, and $Q_6,Q_7$ lie above $Q_1$ and
below $Q_3$, each y-overlapping $Q_b$. (a) $L=71$ is less than $G_7=72$,
$U=72$, $H_5=72$, $\Ppi=74$, and $\Prho=73$, but $\Pchi=71$; the bar records
the chain $0,6,7,b$ that ends the proof of Proposition~\ref{prop:ell5}.
(b) $L=62$ is less than $K=64$, $G_6=63$, $U=63$, $H_5=63$, and $J_6=65$,
but $V=55$; the bar records the chain $\{2,b,6,7\}$ that ends the proof of
Proposition~\ref{prop:ell6}.}
\label{fig:stack}
\end{figure}

\begin{lemma}[Right-hand stack]\label{lem:right-stack}
Suppose $Q_0$ is cornered, the hypotheses of Lemma~\ref{lem:fixed-core}
hold, and the following additional families clear $\Sigma$:
\[
\{0,1,e\}\ (e=6,7),\qquad \{2,3,4,5\},\qquad
\{0,4,5\},\qquad \{1,4,5,e\}\ (e=6,7).
\]
Use the labels $p,c\in\{2,3\}$ from that lemma. Then $Q_4,Q_5$ form a
vertical stack with lower square $Q_b$ and upper square $Q_t$, both lie right
of $Q_c$ and above the original $Q_1$, and
\begin{equation}\label{eq:right-stack-frame}
x_p+s_p\le x_c<s_0\le y_c.
\end{equation}
Square $Q_1$ can be lowered to the floor in the full eight-square packing.
In the resulting packing every square other than $Q_0,Q_1$ lies above $Q_1$.
\end{lemma}
\begin{proof}
Lemma~\ref{lem:fixed-core} gives the relations among the fixed squares
$p,c,4,5$. If $4,5$ were x-disjoint, they and $p,c$ would form a clearing horizontal
quadruple. Thus, $4,5$ x-overlap and form a vertical stack with lower square $b$ and
upper square $t$.

If $y_b\ge s_0$, the vertical chain $0,b,t$ clears $\Sigma$. Hence $y_b<s_0$ and
$x_b\ge s_0$. By Lemma~\ref{lem:fixed-core}, $b$ and $c$ are x-disjoint. If $c$ were
right of $b$, the chain $0,b,c$ would clear $\Sigma$. Thus, $b$ lies right of $c$.
Since $t$ x-overlaps $b$ and is x-disjoint from $c$, it too
lies right of $c$. If $x_c\ge s_0$, then the chain $0,c,b$ would clear $\Sigma$. Hence
$x_c<s_0$. Since $0\le x_c<s_0$, square $c$ x-overlaps $0$,
and hence $y_c\ge s_0$.
Together with $p$ left of $c$, this proves \eqref{eq:right-stack-frame}.

For $m=4,5$, we have $x_m\ge x_c+s_c>x_1$. If $m$ were x-disjoint from
$1$, it would lie right of $1$ and form a clearing horizontal chain
$0,1,m$. Therefore, $m$ x-overlaps $1$ and is y-disjoint from it.
It cannot lie below $1$, because $m,1,c$ would be a vertical chain clearing $\Sigma$. Both
medium squares thus lie above the original $1$.

Squares $0$ and $p$ are x-disjoint from $1$, the second by
\eqref{eq:right-stack-frame} and $x_1\ge s_0$, while $c,4,5$ lie above $1$.
A downward blocker of $1$ can therefore only be $6$ or $7$. Such a blocker $e$
would complete the clearing vertical chain $e,1,4,5$. Thus, no blocker
exists and $1$ can be lowered in the full packing. For every
other index $j\ne0,1$, the horizontal pin to the pair $0,1$, with $1$ now on the floor,
forces x-overlap with at least one of them. It follows that
$y_j\ge s_1$, using $s_0\ge s_1$.
\end{proof}

The first application of Lemma~\ref{lem:right-stack} is a bound by six
side sums.
\begin{proposition}\label{prop:ell5}
For every $s\in\ell_2^\downarrow$,
\[
L_8^*(s)\ge b_4(s):=\min(G_7,\ U,\ H_5,\ \Ppi,\ \Prho,\ \Pchi),
\]
where
\begin{gather*}
G_7=s_0+s_3+s_7,\quad U=s_0+s_4+s_5,\quad H_5=s_1+s_3+s_5,\\
\Ppi=s_1+s_3+s_6+s_7,\quad \Prho=s_1+s_4+s_5+s_7,\quad
\Pchi=s_0+s_5+s_6+s_7.
\end{gather*}
\end{proposition}
\begin{proof}
Set $\Sigma=b_4$ and suppose $L<\Sigma$. Since $G_7>L\ge s_0+s_1$, we have
$s_3+s_7>s_1$. Hence
$Z:=s_2+s_3+s_5+s_7=H_5+(s_2+s_7-s_1)>H_5\ge\Sigma$.
Thus, $\{2,3,4,5\}$ and every $\{2,3,m,e\}$ with $m=4,5$ and $e=6,7$
clear $\Sigma$. Every triple of the four large squares has side sum at least
$H_5$, and every triple $\{0,b,j\}$ with $b=1,2,3$ has side sum at least
$G_7$. These triples clear $\Sigma$ as well, and Lemma~\ref{lem:pane}
corners $0$ in the full packing. The hypotheses of Lemma~\ref{lem:fixed-core}
follow from $G_7,H_5$. The other hypotheses of
Lemma~\ref{lem:right-stack} are supplied by $Z,U,\Prho$. Apply that lemma
and lower $Q_1$ to the floor, keeping the other squares fixed.

We next show that $y_6,y_7<s_0$.
Both $p$ and $c$ lie above $0$: $p$ x-overlaps $0$ by
Lemma~\ref{lem:fixed-core}, and $y_c\ge s_0$ by
\eqref{eq:right-stack-frame}.
For $e=6,7$ we already know $y_e\ge s_1$. If $y_e\ge s_0$, the vertical
pins to $0,p$ and $0,c$ force $e$ to y-overlap both $p,c$, and
hence be x-disjoint from them. If it were also x-disjoint from either
medium square $m$, the quadruple $p,c,m,e$ would be a horizontal chain that clears $\Sigma$.
Thus, it x-overlaps both $4,5$ and is y-disjoint from both.
All three lie above $1$, and the vertical chain $1,4,5,e$ has side sum at
least $\Prho$. Hence $y_e<s_0$ and $x_e\ge s_0$.

By \eqref{eq:right-stack-frame}, $e$ is x-disjoint from $p$. The vertical
pin to the chain $1,4,5$ makes $e$ y-overlap a medium square $m$,
hence be x-disjoint from it. If it were x-disjoint from $c$ as well, the
quadruple $p,c,m,e$ would clear $\Sigma$. Thus, $e$ x-overlaps $c$
and is y-disjoint from it. Since $y_e<s_0\le y_c$, it lies below $c$.

Finally, we find a chain that clears $\Sigma$.
Both $6,7$ lie above $1$ and below $c$. If they were y-disjoint, the
vertical chain $1,6,7,c$ would have sum at least $\Ppi$. Thus, they
y-overlap and are x-disjoint. Let $b,t$ again be the lower and
upper medium squares. Since square $b$ y-overlaps $c$, we have $y_t\ge y_b+s_b>y_c$.
Neither small square, whose top edge is at most $y_c$, can y-overlap $t$. Its pin to $1,4,5$ must therefore be to $b$, making it
x-disjoint from $b$. Moreover it lies left of $b$, since it x-overlaps $c$ and $b$ lies right of $c$. Finally $0,6,7,b$ form a horizontal
chain with side sum at least $\Pchi\ge\Sigma$, a contradiction.
\end{proof}

The second application of Lemma~\ref{lem:right-stack} requires a different
cornering argument. Since the next proof does not assume $G_7\ge\Sigma$,
the pair $Q_3,Q_7$ requires separate treatment.
Figure~\ref{fig:pincer} shows the configuration excluded in the proof of the
next lemma.

\begin{figure}[!htbp]\centering
\begin{minipage}[t]{0.6\textwidth}\centering
\begin{tikzpicture}[x=0.0917cm,y=0.0917cm]
\useasboundingbox (-4.86,-6.21) rectangle (54,54);
\filldraw[fill=black!7,draw=black,line width=.35pt] (22,0) rectangle (50,28);
\node[font=\scriptsize,inner sep=0pt] at (36,14) {$0$};
\filldraw[fill=black!7,draw=black,line width=.35pt] (24,28) rectangle (50,54);
\node[font=\scriptsize,inner sep=0pt] at (37,41) {$1$};
\filldraw[fill=black!7,draw=black,line width=.35pt] (0,30) rectangle (24,54);
\node[font=\scriptsize,inner sep=0pt] at (12,42) {$2$};
\filldraw[fill=black!7,draw=black,line width=.35pt] (0,8) rectangle (22,30);
\node[font=\scriptsize,inner sep=0pt] at (11,19) {$3$};
\filldraw[fill=black!7,draw=black,line width=.35pt] (0,0) rectangle (8,8);
\node[font=\scriptsize,inner sep=0pt] at (4,4) {$4$};
\filldraw[fill=black!7,draw=black,line width=.35pt] (8,0) rectangle (15,7);
\node[font=\scriptsize,inner sep=0pt] at (11.5,3.5) {$5$};
\filldraw[fill=black!7,draw=black,line width=.35pt] (15,0) rectangle (21,6);
\node[font=\tiny,inner sep=0pt] at (18,3) {$6$};
\filldraw[fill=black!7,draw=black,line width=.35pt] (50,0) rectangle (54,4);
\node[font=\tiny,inner sep=0pt] at (52,2) {$7$};
\draw[densely dashed,line width=.3pt] (0,28) -- (54,28);
\node[font=\tiny,inner sep=1pt,left] at (0,28) {$s_0$};
\draw[line width=.7pt] (0,0) rectangle (54,54);
\filldraw[fill=hc!25,draw=hc!70!black,line width=.25pt] (0,-5.4) rectangle (8,-2.16);
\node[font=\tiny,inner sep=0pt] at (4,-3.78) {$4$};
\filldraw[fill=hc!25,draw=hc!70!black,line width=.25pt] (8,-5.4) rectangle (15,-2.16);
\node[font=\tiny,inner sep=0pt] at (11.5,-3.78) {$5$};
\filldraw[fill=hc!25,draw=hc!70!black,line width=.25pt] (22,-5.4) rectangle (50,-2.16);
\node[font=\tiny,inner sep=0pt] at (36,-3.78) {$0$};
\end{tikzpicture}\par
\smallskip{\footnotesize Lemma~\ref{lem:exceptional}, $L=54$\\ $s=(28,26,24,22,8,7,6,4)$}
\end{minipage}
\caption{The double blockage excluded in the proof of
Lemma~\ref{lem:exceptional}; the dashed line is $y=s_0$. Here $L=54$ is less
than $K=56$, $G_6=56$, $H_5=55$, $J_6=56$, and $s_1+s_2+s_3=72$, but not
less than $U=43$ or $V=41$. Square $Q_0$ is on the floor, blocked by $Q_3$
on the left and $Q_7$ on the right. Squares $Q_1,Q_2$ lie above $Q_0$, with
$a=2$ and $b=1$ in the notation of the proof, and $Q_3$ lies below $Q_2$.
Squares $Q_4,Q_5$ lie left of $Q_0$ and below $Q_3$, and x-overlap $Q_3$. The
bar records the horizontal chain $4,5,0$, excluded by $U$.}
\label{fig:pincer}
\end{figure}

\begin{lemma}[Exceptional pincer]\label{lem:exceptional}
Suppose $L<\Sigma\le\min(K,G_6,U,H_5,J_6,V)$, where $K=s_0+s_2+s_7$,
$G_6=s_0+s_3+s_6$, and $J_6=s_1+s_2+s_6$. Then $Q_0$ can be translated to a corner in the full packing.
\end{lemma}
\begin{proof}
Because $s$ is nonincreasing, the following triples clear $\Sigma$:
\begin{align*}
&\{0,1,e\},\ \{0,2,e\} &&\text{for every distinct third index }e,\\
&\{0,3,e\} &&\text{for }e\le6,\ e\ne0,3,\\
&\{1,2,e\} &&\text{for }e\le6,\ e\ne1,2,\\
&\{1,3,m\} &&\text{for }m=4,5.
\end{align*}
Their side sums are at least $K$, $G_6$, $J_6$, and $H_5$, respectively.
Every triple of the four large squares clears $\Sigma$ as well, and so does
the quadruple $2,3,4,5$, whose side sum is at least $V$.

Choose an axis on which $0,1$ are disjoint and move $0$ away from $1$
to the corresponding wall. A blocker would form a clearing triple with
$0,1$. Interchange axes and reflect so that $y_0=0$.
If $Q_0$ has no blocker on its left or no blocker on its right, move it
horizontally to the corresponding corner. Otherwise there is a blocker on
each side.

Neither $1$ nor $2$ can be x-disjoint from $0$: it would form a clearing
horizontal triple with $0$ and the opposite blocker. Thus, $1,2$ x-overlap $0$ and lie
above it. They cannot be y-disjoint, because $0,1,2$ would clear $\Sigma$; they
are therefore x-disjoint.
Square $3$ must be x-disjoint from $0$. Otherwise it lies above $0$,
and its vertical pins to $0,1$ and $0,2$ make it x-disjoint from both
$1,2$, giving the clearing horizontal triple $1,2,3$. Reflect horizontally
so that $3$ lies left of $0$, and let $a,b$ be the left-to-right order
of $1,2$. A blocker with index $j\le6$ on the opposite side would give the
horizontal chain $3,0,j$, which clears $\Sigma$.
Hence the blocker on that side is $7$.

Square $3$ x-overlaps $a$, or else $3,a,b$ form a clearing
horizontal triple. Thus, it is y-disjoint from $a$. It cannot lie above
$a$, or be y-disjoint above $0$, because $0,a,3$ would then be a vertical
chain. Hence $3$ lies below $a$ and y-overlaps $0$.

For $m=4,5$, square $m$ must also y-overlap $0$. Otherwise it
lies above $0$, and pins to $0,a$ and $0,b$ make it x-disjoint from both,
giving the clearing horizontal triple $1,2,m$. Thus, $m$ is x-disjoint
from $0$. Its horizontal pin to $0,3$ forces it to x-overlap
$3$. Hence it lies left of $0$, and it lies left of $b$ because
$x_b\ge x_a+s_a>x_0$. The forbidden triple $a,b,m$ then forces a horizontal
overlap with $a$. Therefore, $m$ is y-disjoint from both $3,a$, and lies
below $a$ because $y_m<s_0\le y_a$.

If $a=1$, the vertical triple $1,3,m$ clears $\Sigma$. Hence $a=2$.
Now $2,3$ are y-disjoint from each other and from both $4,5$. If $4,5$
were y-disjoint, the quadruple $2,3,4,5$ would be a vertical chain that clears $\Sigma$.
Otherwise they y-overlap and are x-disjoint. Both lie left of
$0$, and the horizontal triple $0,4,5$, with side sum $U$, clears $\Sigma$. This excludes the
double blockage and completes the cornering argument.
\end{proof}

\begin{proposition}\label{prop:ell6}
For every $s\in\ell_2^\downarrow$,
\[
L_8^*(s)\ge b_5(s):=\min(K,\ G_6,\ U,\ H_5,\ J_6,\ V).
\]
\end{proposition}
\begin{proof}
Set $\Sigma=b_5$ and suppose $L<\Sigma$. The triple exclusions in
Lemma~\ref{lem:exceptional} hold. In addition, each of the following families has side sum at least $V$:
\begin{equation}\label{eq:theta-quads}
\{2,3,4,5\},\quad \{2,3,m,e\}\ (m=4,5;\ e=6,7),\quad
\{1,4,5,e\}\ (e=6,7),\quad \{1,3,6,7\}.
\end{equation}
For example, the least sums in the last two families are
$s_1+s_4+s_5+s_7\ge V$ and $s_1+s_3+s_6+s_7\ge V$.

Lemma~\ref{lem:exceptional} corners $0$. The hypotheses of
Lemma~\ref{lem:fixed-core} follow from the displayed triple
exclusions, and those of Lemma~\ref{lem:right-stack} follow from
$U$ and \eqref{eq:theta-quads}. Thus, we obtain the right-hand stack,
\eqref{eq:right-stack-frame}, and $Q_1$ lowered to the floor in the full
packing.

We next use square $6$ to identify the labels $p,c$ of the large squares.
All squares other than $0,1$ lie above $1$. The triple $1,2,6$
clears $\Sigma$, and its vertical pin makes $6$ y-overlap $2$, hence be x-disjoint
from it. The pin to $1,4,5$ makes it x-disjoint from a medium square $m$.
If it were also x-disjoint from $3$, the quadruple $2,3,m,6$ would clear $\Sigma$.
Thus, $6$ x-overlaps $3$ and is y-disjoint from it.

If $y_6\ge s_0$, the vertical pins to $0,2$ and $0,3$, from the clearing triples $0,2,6$
and $0,3,6$, make it x-disjoint from both, a contradiction. Hence $y_6<s_0$ and
$x_6\ge s_0$. By \eqref{eq:right-stack-frame}, $6$ is x-disjoint from the
left large square $p$. Since it x-overlaps $3$, we must have
$p=2,c=3$. Moreover $6$ lies below $3$.

Finally, we locate square $7$ and find a chain that clears $\Sigma$.
Square $7$ is x-disjoint from $2$. If $y_7<s_0$, it lies right of $0$,
while $2=p$ lies entirely left of $x=s_0$ by
\eqref{eq:right-stack-frame}. Otherwise the vertical pin to
$0,2$, from the clearing triple $0,2,7$, makes it y-overlap $2$.
The pin to $1,4,5$ makes $7$ x-disjoint from a medium square $m$.
If it were also x-disjoint from $3$, the quadruple $2,3,m,7$ would clear $\Sigma$.
Thus, both $6,7$ x-overlap $3$ and are y-disjoint from it.
They cannot be y-disjoint from each other, since $1,3,6,7$ clears $\Sigma$.
They therefore y-overlap, are x-disjoint, and lie on the same
vertical side of $3$, hence below it, as $6$ does.

Since the lower medium square $b$ y-overlaps $3$, the upper medium
square $t$ satisfies $y_t\ge y_b+s_b>y_3$. Neither small square can
y-overlap $t$. The pins to $1,4,5$ must therefore be to $b$,
making both small squares x-disjoint from it. The horizontal chain
$\{2,b,6,7\}$ has sum at least $V\ge\Sigma$, the final contradiction.
\end{proof}

\subsection{Three bounds from a four-square core}\label{sec:fourcore}
Call the elements
of $\mathcal C=\{1,2,3,4\}$ the core indices and those of
$\mathcal S=\{5,6,7\}$ the small indices. These names are local to
this subsection; the largest square \(Q_0\) is kept separate.

We use two families of chain exclusions. In all applications below they
follow from side-sum inequalities, and hence hold in every packing of the
given side lengths in $[0,L]^2$:
\begin{description}
\item[(C)] No chain contains \(Q_0,Q_b,Q_j\), where \(b\in\mathcal C\) and
\(j\notin\{0,b\}\).
\item[(Q)] No four-square chain contains at least three indices from
\(\mathcal C\).
\end{description}
Figure~\ref{fig:frame8} illustrates the two alternatives in the next lemma.

\begin{figure}[t]\centering
\begin{minipage}[t]{0.32\textwidth}\centering
\begin{tikzpicture}[x=0.0708cm,y=0.0708cm]
\useasboundingbox (0,0) rectangle (63.555,57);
\fill[vc!12] (37,0) rectangle (57,57);
\fill[hc!12] (0,37) rectangle (57,57);
\filldraw[fill=black!7,draw=black,line width=.35pt] (0,0) rectangle (37,37);
\node[font=\scriptsize,inner sep=0pt] at (18.5,18.5) {$0$};
\filldraw[fill=black!7,draw=black,line width=.35pt] (37,0) rectangle (57,20);
\node[font=\scriptsize,inner sep=0pt] at (47,10) {$1$};
\filldraw[fill=black!7,draw=black,line width=.35pt] (38,20) rectangle (57,39);
\node[font=\scriptsize,inner sep=0pt] at (47.5,29.5) {$2$};
\filldraw[fill=black!7,draw=black,line width=.35pt] (39,39) rectangle (57,57);
\node[font=\scriptsize,inner sep=0pt] at (48,48) {$3$};
\filldraw[fill=black!7,draw=black,line width=.35pt] (0,40) rectangle (17,57);
\node[font=\scriptsize,inner sep=0pt] at (8.5,48.5) {$4$};
\filldraw[fill=black!7,draw=black,line width=.35pt] (17,51) rectangle (23,57);
\node[font=\tiny,inner sep=0pt] at (20,54) {$5$};
\filldraw[fill=black!7,draw=black,line width=.35pt] (23,52) rectangle (28,57);
\node[font=\tiny,inner sep=0pt] at (25.5,54.5) {$6$};
\filldraw[fill=black!7,draw=black,line width=.35pt] (28,53) rectangle (32,57);
\node[font=\tiny,inner sep=0pt] at (30,55) {$7$};
\draw[line width=.7pt] (0,0) rectangle (57,57);
\filldraw[fill=vc!25,draw=vc!70!black,line width=.25pt] (59.28,0) rectangle (62.7,20);
\node[font=\tiny,inner sep=0pt] at (60.99,10) {$1$};
\filldraw[fill=vc!25,draw=vc!70!black,line width=.25pt] (59.28,20) rectangle (62.7,39);
\node[font=\tiny,inner sep=0pt] at (60.99,29.5) {$2$};
\filldraw[fill=vc!25,draw=vc!70!black,line width=.25pt] (59.28,39) rectangle (62.7,57);
\node[font=\tiny,inner sep=0pt] at (60.99,48) {$3$};
\end{tikzpicture}\par
\smallskip{\footnotesize (a) Alternative I, $L=57$\\ $s=(37,20,19,18,17,6,5,4)$}
\end{minipage}
\hfill
\begin{minipage}[t]{0.32\textwidth}\centering
\begin{tikzpicture}[x=0.0789cm,y=0.0789cm]
\useasboundingbox (0,-6.555) rectangle (57,57);
\fill[vc!12] (37,0) rectangle (57,57);
\fill[hc!12] (0,37) rectangle (57,57);
\filldraw[fill=black!7,draw=black,line width=.35pt] (0,0) rectangle (37,37);
\node[font=\scriptsize,inner sep=0pt] at (18.5,18.5) {$0$};
\filldraw[fill=black!7,draw=black,line width=.35pt] (37,0) rectangle (57,20);
\node[font=\scriptsize,inner sep=0pt] at (47,10) {$1$};
\filldraw[fill=black!7,draw=black,line width=.35pt] (38,20) rectangle (57,39);
\node[font=\scriptsize,inner sep=0pt] at (47.5,29.5) {$2$};
\filldraw[fill=black!7,draw=black,line width=.35pt] (0,39) rectangle (18,57);
\node[font=\scriptsize,inner sep=0pt] at (9,48) {$3$};
\filldraw[fill=black!7,draw=black,line width=.35pt] (18,40) rectangle (35,57);
\node[font=\scriptsize,inner sep=0pt] at (26.5,48.5) {$4$};
\filldraw[fill=black!7,draw=black,line width=.35pt] (35,39) rectangle (41,45);
\node[font=\tiny,inner sep=0pt] at (38,42) {$5$};
\filldraw[fill=black!7,draw=black,line width=.35pt] (52,39) rectangle (57,44);
\node[font=\tiny,inner sep=0pt] at (54.5,41.5) {$6$};
\filldraw[fill=black!7,draw=black,line width=.35pt] (53,44) rectangle (57,48);
\node[font=\tiny,inner sep=0pt] at (55,46) {$7$};
\draw[line width=.7pt] (0,0) rectangle (57,57);
\filldraw[fill=hc!25,draw=hc!70!black,line width=.25pt] (0,-5.7) rectangle (18,-2.28);
\node[font=\tiny,inner sep=0pt] at (9,-3.99) {$3$};
\filldraw[fill=hc!25,draw=hc!70!black,line width=.25pt] (18,-5.7) rectangle (35,-2.28);
\node[font=\tiny,inner sep=0pt] at (26.5,-3.99) {$4$};
\filldraw[fill=hc!25,draw=hc!70!black,line width=.25pt] (38,-5.7) rectangle (57,-2.28);
\node[font=\tiny,inner sep=0pt] at (47.5,-3.99) {$2$};
\end{tikzpicture}\par
\smallskip{\footnotesize (b) Alternative II, $L=57$\\ $s=(37,20,19,18,17,6,5,4)$}
\end{minipage}
\hfill
\begin{minipage}[t]{0.32\textwidth}\centering
\begin{tikzpicture}[x=0.0592cm,y=0.0592cm]
\useasboundingbox (0,-8.74) rectangle (76,76);
\fill[vc!12] (46,0) rectangle (76,76);
\fill[hc!12] (0,46) rectangle (76,76);
\filldraw[fill=black!7,draw=black,line width=.35pt] (0,0) rectangle (46,46);
\node[font=\scriptsize,inner sep=0pt] at (23,23) {$0$};
\filldraw[fill=black!7,draw=black,line width=.35pt] (46,0) rectangle (76,30);
\node[font=\scriptsize,inner sep=0pt] at (61,15) {$1$};
\filldraw[fill=black!7,draw=black,line width=.35pt] (47,30) rectangle (76,59);
\node[font=\scriptsize,inner sep=0pt] at (61.5,44.5) {$2$};
\filldraw[fill=black!7,draw=black,line width=.35pt] (0,54) rectangle (22,76);
\node[font=\scriptsize,inner sep=0pt] at (11,65) {$3$};
\filldraw[fill=black!7,draw=black,line width=.35pt] (22,55) rectangle (43,76);
\node[font=\scriptsize,inner sep=0pt] at (32.5,65.5) {$4$};
\filldraw[fill=black!7,draw=black,line width=.35pt] (43,59) rectangle (55,71);
\node[font=\scriptsize,inner sep=0pt] at (49,65) {$5$};
\filldraw[fill=black!7,draw=black,line width=.35pt] (55,59) rectangle (66,70);
\node[font=\scriptsize,inner sep=0pt] at (60.5,64.5) {$6$};
\filldraw[fill=black!7,draw=black,line width=.35pt] (66,59) rectangle (76,69);
\node[font=\scriptsize,inner sep=0pt] at (71,64) {$7$};
\draw[line width=.7pt] (0,0) rectangle (76,76);
\filldraw[fill=hc!25,draw=hc!70!black,line width=.25pt] (0,-7.6) rectangle (22,-3.04);
\node[font=\tiny,inner sep=0pt] at (11,-5.32) {$3$};
\filldraw[fill=hc!25,draw=hc!70!black,line width=.25pt] (22,-7.6) rectangle (43,-3.04);
\node[font=\tiny,inner sep=0pt] at (32.5,-5.32) {$4$};
\filldraw[fill=hc!25,draw=hc!70!black,line width=.25pt] (43,-7.6) rectangle (55,-3.04);
\node[font=\tiny,inner sep=0pt] at (49,-5.32) {$5$};
\filldraw[fill=hc!25,draw=hc!70!black,line width=.25pt] (55,-7.6) rectangle (66,-3.04);
\node[font=\tiny,inner sep=0pt] at (60.5,-5.32) {$6$};
\filldraw[fill=hc!25,draw=hc!70!black,line width=.25pt] (66,-7.6) rectangle (76,-3.04);
\node[font=\tiny,inner sep=0pt] at (71,-5.32) {$7$};
\end{tikzpicture}\par
\smallskip{\footnotesize (c) Alternative II, $L=76$\\ $s=(46,30,29,22,21,12,11,10)$}
\end{minipage}
\caption{The frame alternatives of
Lemma~\ref{lem:fourcore-frame}; the shaded regions are $x\ge s_0$ and
$y\ge s_0$. In (a) and (b), $L=57$ is less than $s_0+s_4+s_7=58$ and
$s_2+s_3+s_4+s_7=58$, so (C) and (Q) hold. (a) Alternative~I with $t=3$ and
$k=4$; the bar records the right-arm chain $1,2,3$. (b) Alternative~II with
$b=1$, $t=2$, $p=3$, and $q=4$; the bar records the horizontal chain $p,q,t$.
(c) Alternative~II with $L=76$ less than $s_0+s_4+s_7=77$,
$s_2+s_3+s_4+s_7=82$, and $s_1+s_2+s_6+s_7=80$, so no chain contains $b,t$
and two small squares. The small squares form a horizontal chain above
$Q_t$, and $Q_6,Q_7$ lie right of $Q_p,Q_q$. The bar records the chain
$3,4,5,6,7$ that ends the proofs of Propositions~\ref{prop:ell8}
and~\ref{prop:ell9}.}
\label{fig:frame8}
\end{figure}
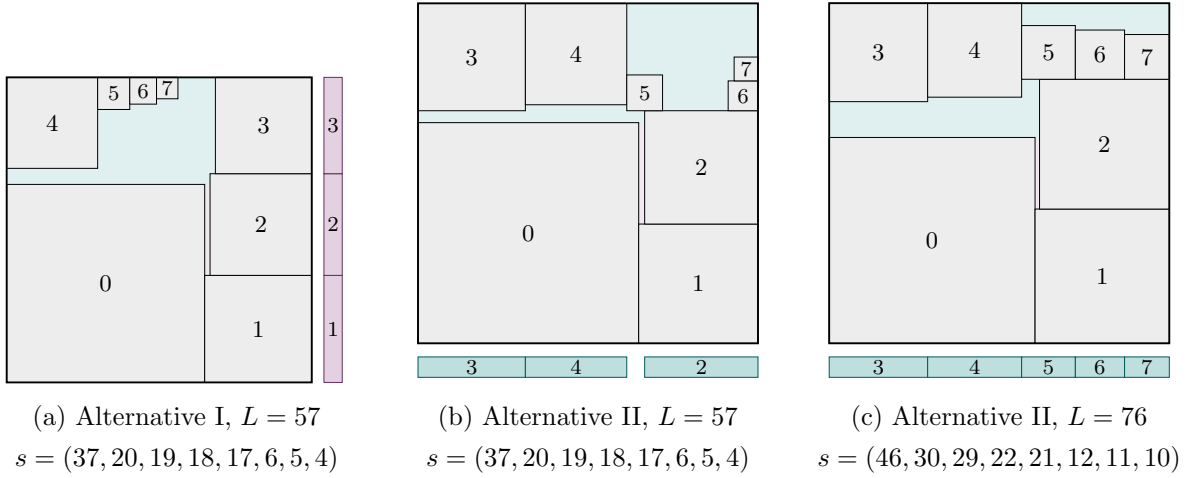

\begin{lemma}[Cornering and the two frame alternatives]\label{lem:fourcore-frame}
Suppose (C) and (Q) hold for every packing of the given side lengths in
$[0,L]^2$. Then $Q_0$ can be translated to a corner while all other squares
remain fixed. After reflections set $Q_0=[0,s_0]^2$,
$\mathcal R=\{i\ge1:x_i\ge s_0\}$, and
$\mathcal U=\{i\ge1:y_i\ge s_0\}$. Up to interchanging axes, one of
the following holds.

\begin{enumerate}[label=\textup{\Roman*.},ref=\Roman*]
\item Three core squares form the right-arm vertical chain. Let \(t\) be its
topmost index and \(k\) the remaining core index. All three small squares
lie in \(\mathcal U\setminus\mathcal R\). Squares \(Q_k,Q_t\) are x-disjoint
from each other and from all three small squares.
\item Write the right-arm core squares as \(b,t\), bottom to top, and the
upper-arm core squares as \(p,q\), left to right. The axes may be chosen
so that \(Q_t\) lies right of \(Q_q\). Then \(p,q,t\) is a horizontal
three-core chain. Every small square x-overlaps \(Q_t\) and is
y-disjoint from both \(Q_b,Q_t\).
\end{enumerate}
If, in Alternative II, no chain can contain \(b,t\) and two small
indices, then the three small squares y-overlap pairwise and form
a horizontal chain. Its middle and rightmost squares are x-disjoint from
both \(Q_p,Q_q\). Moreover, either all three small squares belong
to \(\mathcal R\), or all three lie above \(Q_t\).
\end{lemma}
\begin{proof}
Choose an axis on which \(Q_0,Q_1\) are disjoint. Move \(Q_0\) away from \(Q_1\)
toward the corresponding wall. A blocker would make a chain with \(Q_0,Q_1\),
contrary to (C). Interchange axes and reflect so that the wall reached is the
bottom wall, \(y_0=0\).

If \(Q_0\) has no blocker on its left or no blocker on its right, move it
horizontally to the corresponding corner. Otherwise there is a blocker on
each side. Any core square x-disjoint
from \(Q_0\) would form a chain with \(Q_0\) and a blocker on the opposite side.
By (C), every core square therefore x-overlaps \(Q_0\), and hence
lies above it. Any two core squares that were y-disjoint would form a chain
with \(Q_0\), again contrary to (C). The four core squares consequently
y-overlap pairwise and form a horizontal chain, contradicting (Q).

After cornering, the sets \(\mathcal R\) and \(\mathcal U\) cover
\(\{1,\ldots,7\}\). A core square in \(\mathcal R\) x-overlaps every other square in \(\mathcal R\): otherwise those two squares
and \(Q_0\) violate (C). Thus, it is y-disjoint from all of them. Symmetrically,
a core square in \(\mathcal U\) y-overlaps every other square in
\(\mathcal U\), and is x-disjoint from it. Unlike in
Lemma~\ref{lem:full-frame}, the arms need not be chains: (C) gives no
exclusion for a triple consisting of \(Q_0\) and two small squares.

Let
\(r=|\mathcal C\cap\mathcal R|\) and
\(u=|\mathcal C\cap\mathcal U|\).
The core squares in either arm form a chain. By (Q), \(r,u\le3\).
Also at most one core square belongs to both arms: two such squares would
have to overlap on both axes. Since the arms cover the four core squares,
up to interchanging axes there are just two possibilities: one arm contains
three core squares, or $r=u=2$ with disjoint core arms.

In Alternative I, a small square in \(\mathcal R\) would be y-disjoint from
all three core squares there, giving a chain excluded by (Q). Thus, all
small squares lie in \(\mathcal U\setminus\mathcal R\). The index \(k\) belongs
to \(\mathcal U\). Hence it is x-disjoint from all the small squares.

If \(t\in\mathcal U\), the same is true of \(t\), and \(k,t\) are
x-disjoint. Otherwise \(y_t<s_0\). Any small square, or \(Q_k\), that
x-overlapped \(Q_t\) would have to lie above it: its lower edge
is at least \(s_0>y_t\). Together with the three right-arm core squares,
it would make a chain excluded by (Q). Thus, \(Q_t\) is x-disjoint from
\(Q_k\) and all three small squares.

In Alternative II, each core square belongs to exactly one arm. Hence
\(x_p,x_q<s_0\le x_b,x_t\) and \(y_b,y_t<s_0\le y_p,y_q\).
The disjoint squares \(Q_q,Q_t\) satisfy either ``\(t\) right of
\(q\)'' or ``\(q\) above \(t\).'' Interchanging the axes exchanges the two
arms, and hence exchanges the roles \(b,t\) with \(p,q\), respectively; it
turns the second alternative into the first. The square indices themselves
are unchanged. Thus,
\(p,q,t\) is a horizontal chain.

A small square in \(\mathcal R\) x-overlaps both right-arm core
squares, and hence is y-disjoint from both. A small square outside
\(\mathcal R\) belongs to \(\mathcal U\); it is x-disjoint from \(p,q\).
It cannot also be x-disjoint from \(t\), by (Q). It therefore x-overlaps \(t\), is y-disjoint from it, and lies above it because its lower
edge is at least \(s_0>y_t\). It also lies above \(b\). This proves the
assertions about all three small squares.

Under the additional exclusion in the lemma, two small squares cannot be
y-disjoint, since together with \(b,t\) they would form a forbidden
chain. Thus, the small squares y-overlap pairwise and are mutually
x-disjoint. Order them from left to right as \(e_1,e_2,e_3\). Since \(e_1\)
x-overlaps \(t\), we have
\(x_{e_2}\ge x_{e_1}+s_{e_1}>x_t\ge x_q+s_q\). Thus, \(e_2,e_3\) lie right of \(q\), and also right of \(p\).
Finally, if one small square is outside \(\mathcal R\), it lies above \(t\)
by the preceding argument. Since all three are y-disjoint from \(t\) and are pairwise y-overlapping, all must lie on the same side of it, namely above.
\end{proof}

In either frame alternative, three core squares form a chain, and hence
$L\ge s_2+s_3+s_4$: use the right-arm triple in Alternative I or
$p,q,t$ in Alternative II.

\begin{proposition}\label{prop:ell7}
For every $s\in\ell_2^\downarrow$,
\[
L_8^*(s)\ge b_6(s):=\min(E,\ X),
\]
where $X=s_3+s_4+s_6+s_7$.
\end{proposition}
\begin{proof}
Put \(\Sigma=b_6\) and suppose \(L<\Sigma\).

Every triple in (C) has side sum at least \(E\). Every four-square family
containing at least two core indices has side sum at least \(X\): for three
or four core indices the additional core sides only increase this minimum.
Thus, (C) and (Q) hold, and we may use the frame.

The frame lemma also gives \(L\ge s_2+s_3+s_4\). Since \(X>L\),
it follows that \(s_6+s_7>s_2\ge s_4\).
Consequently,
\begin{equation}\label{eq:PfromX}
\Pchi-E=s_5+s_6-s_4\ge s_6+s_7-s_4>0,
\end{equation}
and \(\Pchi>E\ge\Sigma\).

In Alternative I, any two x-disjoint small squares, together with
\(k,t\), would make a clearing four-square chain. Otherwise all three
small squares are pairwise x-overlapping and are a vertical chain above
\(Q_0\). With \(Q_0\), their side sum is \(\Pchi\), another contradiction.

In Alternative II, a chain consisting of \(b,t\) and two small squares
has side sum at least \(X\). The frame lemma therefore produces the horizontal
chain consisting of \(p,q\) and the middle and rightmost small
squares. This chain also has side sum at least \(X\). Both alternatives are impossible.
\end{proof}

\begin{proposition}\label{prop:ell8}
For every $s\in\ell_2^\downarrow$,
\[
L_8^*(s)\ge b_7(s):=\min(E,\ J_4,\ Y),
\]
where $Y=s_2+s_4+s_6+s_7$.
\end{proposition}
\begin{proof}
Put \(\Sigma=b_7\) and suppose \(L<\Sigma\). Call the core indices
\(1,2\) marked. The following families clear $\Sigma$:
\begin{enumerate}
\item the triples in (C), with side sum at least \(E\);
\item a triple of core indices containing both marked indices, with side sum
at least \(J_4\);
\item four-square families having at least three core indices, with side
sum at least \(Y\);
\item two core indices, at least one marked, and two small indices, with side
sum at least \(Y\).
\end{enumerate}
For (3), the least sum with exactly three core indices is
\(s_2+s_3+s_4+s_7\ge Y\). For (4), the least sum is precisely \(Y\).
In particular (C) and (Q) hold.

Again the frame gives a three-core chain. Hence
\(L\ge s_2+s_3+s_4\), and \(Y>L\) gives \(s_6+s_7>s_3\ge s_4\).
It follows, as in \eqref{eq:PfromX}, that \(\Pchi>E\ge\Sigma\).
Lemma~\ref{lem:chain} gives \(L\ge s_0+s_1\), and \(E>L\) gives
\(s_3+s_5\ge s_4+s_7>s_1\ge s_2\).
Therefore, the five-square family with indices \(3,4,5,6,7\) clears $\Sigma$:
\begin{equation}\label{eq:fiveY}
s_3+s_4+s_5+s_6+s_7=Y+(s_3+s_5-s_2)>Y\ge\Sigma.
\end{equation}

In Alternative I, the right-arm triple cannot contain both marked indices.
Thus, the outside core index \(k\) is marked. If two small squares are
x-disjoint, they and \(k,t\) make a clearing chain of type (4). Otherwise
the small squares and \(Q_0\) form a vertical chain with side sum
\(\Pchi\).

In Alternative II, the horizontal triple \(p,q,t\) cannot contain
both marked indices. Hence \(b\) is marked. Since type (4) forbids a chain
consisting of \(b,t\) and two small squares, the last conclusions
of Lemma~\ref{lem:fourcore-frame} apply. If either \(p\) or \(q\) is marked,
the resulting chain with the middle and rightmost small squares is of
type (4).

It remains that \(\{p,q\}=\{3,4\}\) and
\(\{b,t\}=\{1,2\}\). If all small squares belong to \(\mathcal R\), their
horizontal chain together with \(Q_0\) has side sum \(\Pchi\).
Otherwise, by the last assertion of Lemma~\ref{lem:fourcore-frame}, all three
lie above \(Q_t\). Since \(Q_t\) lies above \(Q_b\) and \(\{b,t\}=\{1,2\}\),
its top edge satisfies
\(y_t+s_t\ge s_1+s_2\ge s_2+s_4>s_0\), where the final inequality follows from \(J_4>L\ge s_0+s_1\).
Thus, all small squares belong to \(\mathcal U\), and each is x-disjoint from
\(p,q\). With their mutual x-disjointness this gives the five-square
horizontal chain \(3,4,5,6,7\), which clears $\Sigma$ by \eqref{eq:fiveY}.
\end{proof}

\begin{proposition}\label{prop:ell9}
For every $s\in\ell_2^\downarrow$,
\[
L_8^*(s)\ge b_8(s):=\min(E,\ H_4,\ M,\ \Pchi),
\]
where $H_4=s_1+s_3+s_4$ and $M=s_1+s_4+s_6+s_7$.
\end{proposition}
\begin{proof}
Put \(\Sigma=b_8\) and suppose \(L<\Sigma\).
The triples in (C) have side sum at least \(E\). Every triple of core
indices containing index \(1\) has side sum at least \(H_4\). Since
\(E>L\ge s_0+s_1\), we have \(s_2+s_7\ge s_4+s_7>s_1\). Hence every
three-core-plus-one-small family clears $\Sigma$, because
\(s_2+s_3+s_4+s_7=H_4+(s_2+s_7-s_1)>H_4\ge\Sigma\).
The four-core family clears $\Sigma$ as well. Thus, (C) and (Q) hold.
Finally any family consisting of index \(1\), another core index, and two
small indices has side sum at least \(M\).

In Alternative I, the right-arm triple cannot contain index \(1\). Hence
\(k=1\). Two x-disjoint small squares, together with \(k,t\), would form a chain
with side sum at least \(M\). Otherwise the small squares and \(Q_0\) form a
vertical chain with side sum \(\Pchi\).

In Alternative II, the horizontal triple \(p,q,t\) cannot contain
index \(1\). Hence \(b=1\). A chain consisting of \(b,t\) and two small
squares has side sum at least \(M\). Thus, the small squares y-overlap pairwise
and form a horizontal chain. If all belong to \(\mathcal R\), this chain with
\(Q_0\) has side sum \(\Pchi\).

Otherwise, by the last assertion of Lemma~\ref{lem:fourcore-frame}, all three
lie above \(Q_t\). The inequality \(H_4>L\ge s_0+s_1\) implies
\(s_3+s_4>s_0\). Since \(Q_t\) lies above \(Q_b=Q_1\), we have
\(y_t+s_t\ge s_1+s_t\ge s_1+s_4\ge s_3+s_4>s_0\).
All small squares therefore belong to \(\mathcal U\), and are x-disjoint
from \(p,q\). The resulting five-square horizontal chain has sum
$s_p+s_q+s_5+s_6+s_7\ge s_3+s_4+s_5+s_6+s_7>s_0+s_5+s_6+s_7=\Pchi\ge\Sigma$.
This is the final contradiction.
\end{proof}

\section{Proof of \texorpdfstring{Theorem~\ref{thm:main}}{Theorem 1.1}}\label{sec:proof}
\subsection{The lower bounds imply the formula}\label{sec:combine}
We now use only the formulas~\eqref{eq:intro-six} and~\eqref{eq:intro-seven}
for $L_6^*$ and $L_7^*$ from~\cite[Theorems~3.8 and~5.5]{BadgerSeven},
the definition of $\F$, the column bound, and the lower bounds
$b_1,\ldots,b_8$. For $n=8$, Lemma~\ref{lem:column} gives
$L_8^*(s)\ge\min(D,P)$ for all $s\in\ell_2^\downarrow$, where $D=s_1+s_3+s_7$ and $P=s_1+s_4+s_5+s_6+s_7$, the value of $P_8$.
Using the notation of Table~\ref{tab:names}, the seven-square
formula~\eqref{eq:intro-seven} reads
\begin{equation}\label{eq:L7}
L_7^*=\max\bigl(L_6^*,\min(\max(W,H),\max(C_0,C_3),R)\bigr).
\end{equation}
Table~\ref{tab:bounds} collects the lower bounds from
Section~\ref{sec:lower}.

\begin{center}
\captionsetup{type=table}
\small
\renewcommand{\arraystretch}{1.12}
\caption{The eight lower bounds of Section~\ref{sec:lower}.}\label{tab:bounds}
\begin{tabular}{@{}l@{\qquad}l@{\qquad}l@{}}
\toprule
$b_1=\min(T,\KK)$ & $b_4=\min(G_7,U,H_5,\Ppi,\Prho,\Pchi)$ & $b_7=\min(E,J_4,Y)$\\
$b_2=\min(C_0,E,J_4,J_{67})$ & $b_5=\min(K,G_6,U,H_5,J_6,V)$ & $b_8=\min(E,H_4,M,\Pchi)$\\
$b_3=\min(J_7,V)$ & $b_6=\min(E,X)$ & \\
\bottomrule
\end{tabular}
\end{center}

\begin{lemma}[Five balanced partitions]\label{lem:five-partitions}
For every $s\in\ell_2^\downarrow$, the minimum defining $\KK(s)$ in
Proposition~\ref{prop:ell1} is attained by one of the partitions
\[
\begin{array}{c|c|c}
& I&J\\\hline
\xi&\{1,3,7\}&\{2,4,5\}\\
\sigma&\{1,4,7\}&\{2,3,5\}\\
\tau&\{1,4,5\}&\{2,3,7\}\\
\upsilon&\{1,5,7\}&\{2,3,4\}\\
\varphi&\{1,2,7\}&\{3,4,5\}.
\end{array}
\]
The labels indicate the corresponding rows of Table~\ref{tab:pieces}. Each
partition supplies two terms of that row after $s_6$ is added to its
part-sums. These are all the terms of $\sigma$, $\tau$, and $\upsilon$; the
rows $\xi$ and $\varphi$ also have the terms $H_5$ and $H$,
respectively, treated in Lemma~\ref{lem:threshold}.
\end{lemma}
\begin{proof}
Identify each of the ten unordered $3+3$ partitions by its part containing
$1$. The $\varphi$ partition has no larger maximum part-sum than any
partition whose part containing $1$ is $\{1,2,3\}$, $\{1,2,4\}$, or
$\{1,2,5\}$. Indeed, both of its part-sums are at most $s_1+s_2+s_5$, which
is at most the sum over each of these three parts. Similarly, the $\xi$
partition has no larger maximum part-sum than any partition whose part
containing $1$ is $\{1,3,4\}$ or $\{1,3,5\}$, since both of its part-sums are
at most $s_1+s_3+s_5$. These comparisons use only $s_0\ge\cdots\ge s_7$.
\end{proof}

\begin{lemma}[Truncation at $T$]\label{lem:threshold}
If $\F\le T$, then $\max(L_7^*,b_1)\ge\F$.
\end{lemma}
\begin{proof}
If $L_7^*\ge T$, then $\F=L_7^*=T$. Suppose $L_7^*<T$.
Since $L_7^*\ge s_0+s_1$, we have $s_1<s_6+s_7$. The additional terms $H_5$
in $\xi$ and $H$ in $\varphi$ are redundant:
$H_5\le s_2+s_4+s_5+s_6$ and $H\le J_{67}$.
For the first comparison use $s_2\ge s_3$ and
$s_4+s_6\ge s_6+s_7>s_1$. For the second use
$s_5\le s_1<s_6+s_7$. Thus, by Lemma~\ref{lem:five-partitions},
$\min(\xi,\sigma,\tau,\upsilon,\varphi)=\KK$.

Every remaining maximum except $\psi$ has a term at
least $T$. In Table~\ref{tab:pieces}, these terms are $C_0$ for
$\alpha,\beta,\gamma,\varepsilon,\theta,\iota,\kappa,\mu,\pi,\rho$;
$U$ for $\delta,\lambda$; $G_6$ for $\eta$; $T$ for $\zeta,\nu$;
$E$ for $\omicron$; and $\Pchi$ for $\chi$.
Here $\psi=P$ is no smaller than $\tau$.
Indeed $R<P$, while
$P-(s_2+s_3+s_6+s_7)=(s_1-s_2)+(s_4+s_5-s_3)>0$ because $s_4+s_5\ge s_6+s_7>s_1\ge s_3$.
Consequently, $\min(T,\min_\omega\omega)=\min(T,\KK)$.
Since $\F=\max(L_7^*,\min_\omega\omega)$ and $L_7^*<T$, it follows that
\begin{equation}\label{eq:truncation-arithmetic}
\min(\F,T)=\max\Bigl(L_7^*,\min\Bigl(T,\min_\omega\omega\Bigr)\Bigr)=\max(L_7^*,b_1).
\end{equation}
The left-hand side equals $\F$ because $\F\le T$.
\end{proof}

\begin{theorem}\label{thm:dominate}
For every ordered positive eight-tuple,
\[
\mathcal L(s):=\max\{L_7^*,\min(D,P),b_1,\ldots,b_8\}\ge\F(s).
\]
\end{theorem}
\begin{proof}
Suppose instead that $\mathcal L<\F$. Then $L_7^*<\F$. Hence $\F$ equals the
minimum of the side-sum maxima and every such maximum is at least $\F$. In particular,
$P=\psi\ge\F$.

We first use the column bound and the bounds $b_1$ and $b_2$.
Lemma~\ref{lem:threshold} gives $T<\F$. Since $\min(D,P)<\F$,
we also have $D<\F$. Comparing with $\nu=\max(T,D,R)$ gives $R\ge\F$.
Moreover $C_3<\F$; indeed $H\ge C_3$ because $s_2\ge s_3$ and $s_5\ge s_6$, and~\eqref{eq:L7} gives
$L_7^*\ge\min(C_3,R)$. Since $L_7^*<\F\le R$, this gives $C_3<\F$.

If $C_0\ge\F$, then the entries $\max(C_0,C_3)$ and $R$ of the
minimum in~\eqref{eq:L7} are at least $\F$. Thus, $L_7^*<\F$ forces $W,H<\F$.
Comparison with $\zeta=\max(W,J_4,T)$, $\omicron=\max(W,E,H)$,
and $\varphi=\max(W,J_{67},H)$ gives $J_4,E,J_{67}\ge\F$, respectively. Then $b_2\ge\F$, a contradiction.
Consequently, $\max(L_7^*,C_0,C_3)<\F\le R$.
The three maxima $\pi,\rho,\chi$ are, respectively,
$\max(C_0,\Ppi)$, $\max(C_0,C_3,\Prho)$, and
$\max(C_3,\Pchi)$. Since each is at least $\F$, we conclude that
\begin{equation}\label{eq:P-auto}
\Ppi,\Prho,\Pchi\ge\F.
\end{equation}

Next we use the bounds $b_3$, $b_4$, and $b_5$.
Suppose $H_5\ge\F$. If $U<\F$, comparison with $\delta=\max(U,J_7,C_3)$ and
$\lambda=\max(V,C_3,U)$ gives $J_7,V\ge\F$, contradicting $b_3<\F$. Thus, $U\ge\F$.
If $G_7\ge\F$, then \eqref{eq:P-auto} gives $b_4\ge\F$, also impossible.
Hence $G_7<\F$. Comparison with $\beta=\max(V,C_0,G_7,C_3)$ and
$\gamma=\max(C_0,G_7,J_6)$ gives $V,J_6\ge\F$. If $J_7\ge\F$, then $b_3\ge\F$, again a contradiction. Otherwise $J_7<\F$, and the inequality
$\eta=\max(G_6,J_7)\ge\F$ gives $G_6\ge\F$. Since $C_0,C_3<\F$, the
inequality $\theta=\max(K,C_0,C_3)\ge\F$ gives $K\ge\F$.
Now $b_5\ge\F$, a contradiction. Therefore, $H_5<\F$.

Finally, we use the bounds $b_6$, $b_7$, and $b_8$.
If $E<\F$, then comparison with $\alpha=\max(J_7,C_0,E,H_5)$ and
$\iota=\max(V,C_0,H_5,E)$ gives $J_7\ge\F$ and $V\ge\F$, contradicting $b_3<\F$. Thus,
$\max(L_7^*,C_0,H_5)<\F\le E$. The remaining comparisons are
\[
\begin{aligned}
\kappa&=\max(C_0,J_4,X),&\qquad
\varepsilon&=\max(C_0,H_4,Y),\\
\mu&=\max(C_0,H_5,M),&
\chi&=\max(C_3,\Pchi).
\end{aligned}
\]
Hence $\max(J_4,X)\ge\F$, $\max(H_4,Y)\ge\F$, and $M,\Pchi\ge\F$.
If $X\ge\F$, then $b_6\ge\F$. Otherwise $J_4\ge\F$.
If $Y\ge\F$, then $b_7\ge\F$. Otherwise $H_4\ge\F$, and
$b_8\ge\F$. Each alternative contradicts $\mathcal L<\F$.
\end{proof}

\begin{proof}[Proof of the formula in Theorem~\ref{thm:main}]
Fix $s\in\ell_2^\downarrow$. Deletion of $Q_7$ gives $L_8^*\ge L_7^*$,
Lemma~\ref{lem:column} gives $L_8^*\ge\min(D,P)$, and
Section~\ref{sec:lower} gives $L_8^*\ge b_i$ for $i=1,\ldots,8$. Since
$L_7^*$, $\min(D,P)$, $b_1,\ldots,b_8$, and $\F$ are continuous,
Theorem~\ref{thm:dominate} extends from positive tuples to all of
$\ell_2^\downarrow$. Therefore, $L_8^*\ge\mathcal L\ge\F$.
Proposition~\ref{prop:eight-upper} gives the reverse inequality
$L_8^*\le\min_\omega\ell_\omega=\F$, which proves both identities.
\end{proof}

\subsection{Minimality of convex representations}\label{sec:minimality}
Every strategy cost $\ell_\sigma$ is convex and positively homogeneous. In
particular it is subadditive:
$\ell_\sigma(s+t)\le\ell_\sigma(s)+\ell_\sigma(t)$.
The minimum $L_n^*(s)$ need not be subadditive in $s$. A strict reverse
inequality $L_n^*(s+t)>L_n^*(s)+L_n^*(t)$ shows that no single strategy is
optimal at both $s$ and $t$: if $\ell_\sigma(s)=L_n^*(s)$ and
$\ell_\sigma(t)=L_n^*(t)$, then
$L_n^*(s+t)\le\ell_\sigma(s+t)\le L_n^*(s)+L_n^*(t)$.
\begin{lemma}[Counting convex costs, {\cite[Lemma~6.1]{BadgerSeven}}]\label{lem:count}
Let $D$ be a convex cone and $F:D\to\mathbb R$ positively homogeneous.
Suppose $p_1,\ldots,p_N\in D$ satisfy
\begin{equation}\label{eq:pairwise-super}
F(p_i+p_j)>F(p_i)+F(p_j)\quad\text{for all }i<j.
\end{equation}
Every representation $F=\min_{1\le k\le m}h_k$ with convex $h_k$ on $D$
requires $m\ge N$.
\end{lemma}
\begin{proof}[Proof of the bound $m\ge23$ in Theorem~\ref{thm:main}]
The twenty-three strategies of Theorem~\ref{thm:main} form a complete
family, and their costs are convex. For the lower bound, use the twenty-three integer vectors in
Table~\ref{tab:cert}. Substitution into the formula of
Theorem~\ref{thm:main} at those vectors and their $253$ pairwise sums gives
\begin{equation}\label{eq:eight-super}
L_8^*(p_\omega+p_{\omega'})>L_8^*(p_\omega)+L_8^*(p_{\omega'})\quad\text{for all }\omega\ne\omega'.
\end{equation}
For instance, Table~\ref{tab:cert} gives
$L_8^*(p_\alpha)=L_8^*(p_\beta)=41$, while direct substitution yields
$L_8^*(p_\alpha+p_\beta)=83>82$. All $253$ differences are positive
integers. These are finite integer computations with the formula of
Theorem~\ref{thm:main}. The accompanying code~\cite{BadgerCode} verifies
these computations and checks its data against the manuscript tables. Lemma~\ref{lem:count}
requires at least twenty-three convex functions, hence at least
twenty-three strategies.
\end{proof}

\begin{table}[t]
\centering
\footnotesize
\renewcommand{\arraystretch}{1.05}
\setlength{\tabcolsep}{4pt}
\begin{tabular}{clc@{\qquad}clc}
\toprule
$\omega$ & \multicolumn{1}{c}{$p_\omega$} & $L_8^*(p_\omega)$ & $\omega$ & \multicolumn{1}{c}{$p_\omega$} & $L_8^*(p_\omega)$\\
\midrule
$\alpha$ & $(21,18,15,13,12,10,9,8)$ & $41$ & $\nu$ & $(195,194,162,161,66,65,64,6)$ & $389$ \\
$\beta$ & $(21,19,16,13,11,10,8,7)$ & $41$ & $\xi$ & $(71,32,31,27,26,24,23,22)$ & $104$ \\
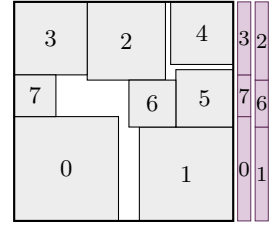
$\gamma$ & $(20,18,15,14,12,11,9,8)$ & $42$ & $\omicron$ & $(225,166,128,100,98,97,96,68)$ & $391$ \\
$\delta$ & $(85,83,69,67,53,52,39,38)$ & $190$ & $\pi$ & $(39,27,24,23,19,18,9,8)$ & $67$ \\
$\varepsilon$ & $(61,38,33,32,31,20,19,18)$ & $101$ & $\rho$ & $(38,29,26,25,16,15,11,8)$ & $68$ \\
$\zeta$ & $(690,413,412,279,278,277,269,144)$ & $1103$ & $\sigma$ & $(73,34,31,28,27,25,24,23)$ & $108$ \\
$\eta$ & $(164,162,131,128,98,97,66,65)$ & $358$ & $\tau$ & $(527,212,211,181,177,176,175,173)$ & $740$ \\
$\theta$ & $(71,69,54,52,50,37,35,32)$ & $157$ & $\upsilon$ & $(77,38,31,30,29,27,26,25)$ & $116$ \\
$\iota$ & $(50,47,38,30,29,21,20,19)$ & $98$ & $\varphi$ & $(1034,466,433,369,368,367,366,236)$ & $1501$ \\
$\kappa$ & $(952,580,579,478,477,342,341,340)$ & $1636$ & $\chi$ & $(98,87,66,65,54,35,27,26)$ & $186$ \\
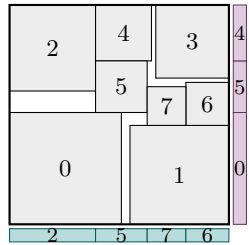
$\lambda$ & $(26,23,20,17,13,12,10,9)$ & $51$ & $\psi$ & $(265,264,231,230,68,67,66,65)$ & $530$ \\
$\mu$ & $(38,27,23,22,20,17,10,9)$ & $66$ &  & &  \\
\bottomrule
\end{tabular}
\medskip
\caption{Ordered tuples $p_\omega$ with $L_8^*(p_\omega+p_{\omega'})>L_8^*(p_\omega)+L_8^*(p_{\omega'})$ for all $\omega\ne\omega'$. At $p_\omega$, the strategy $\sigma_\omega$ is the only one of the 23 with cost $L_8^*(p_\omega)$.}
\label{tab:cert}
\end{table}

\begin{corollary}[Unit-container parameter sets]\label{cor:unit}
The set $X_8=\{s\in\ell_2^\downarrow:L_8^*(s)\le1\}$ is a union of
twenty-three convex subsets and not of fewer.
\end{corollary}
\begin{proof}
By Theorem~\ref{thm:main}, $X_8$ is the union of the sublevel sets
$\{\ell_\omega\le1\}$ of the twenty-three convex strategy costs. For the
last assertion, put $\hat p_\omega=p_\omega/L_8^*(p_\omega)\in X_8$ for the
vectors $p_\omega$ in Table~\ref{tab:cert}. For $\omega\ne\omega'$, the point
$(p_\omega+p_{\omega'})/(L_8^*(p_\omega)+L_8^*(p_{\omega'}))$ lies on the
segment joining $\hat p_\omega$ and $\hat p_{\omega'}$, and
by~\eqref{eq:eight-super} and positive homogeneity its optimal value is
greater than one. Thus, no convex subset of $X_8$ contains two of the
points $\hat p_\omega$, whereas one of fewer than twenty-three sets
covering $X_8$ would contain two of them.
\end{proof}

\section*{Development of the results and AI assistance}
I found the candidate formula for $L_8^*$ in the second half of
July 2026, using Claude Opus~4.8 (Anthropic) as a research assistant.
The search was guided by the form $L_n^*=\max(L_{n-1}^*,\cdot)$ of the
formulas for $n\le7$. The AI system computed exact optima for integer side
vectors with the constraint-programming solver CP-SAT of Google OR-Tools,
and compared them with the minimum of the costs of the strategies found so
far. A side vector with $L_7^*<L_8^*$ at which this minimum exceeds $L_8^*$
shows that a strategy is missing. I first restricted the searches to the
region $L_8^*<s_4+s_5+s_6+s_7$, where no four squares can form a chain, so
that every chain determining the container side has at most three squares.
They produced such side vectors, and new strategies were
read off their optimal packings until exhaustive searches over integer side
vectors with $s_0\le40$ found none there.

Outside that region, I inserted the eighth square into the three
seven-square arrangements of~\cite{BadgerSeven} at geometrically
reasonable places and read candidate terms off the chains that determine
the container side. This gave six further strategies. The search then
became a loop. I chose pairs of side vectors at which the current strategies
failed, the AI system computed and drew their optimal packings, and I
proposed a single arrangement serving both. The AI system then checked the
proposal against exact optima at all nonincreasing integer side vectors
with $s_0\le17$ and at random side vectors, and occasionally repaired a
term. Extracting the separation patterns and chain sums from the
optimal packings required my intervention; on one occasion the AI system
read a term off the arithmetic rather than the geometry, and I read the
correct term off the picture. To normalize the scale, I ran the later
searches on slices with $s_0$ fixed, and a strategy was added only after a
side vector was found at which it is needed. I found the last
strategy, $\sigma_\psi$, associated with the five-term sum $s_1+s_4+s_5+s_6+s_7$,
by scaling and perturbing a side vector with four large and four small
squares.

By the end of July, an exhaustive search over strictly decreasing positive integer side
vectors with $s_0=44$ (about $32$ million side vectors and $450{,}000$ exact solves) found no side vector
at which the twenty-three strategies fail to attain $L_8^*$, and each
strategy was the only one attaining $L_8^*$ at some side vector there. The
Greek letters order the strategies by the number of such side vectors, from
$\alpha$ (most) to $\psi$ (fewest).

Beginning in late July, I attempted the lower bound one strategy
at a time, starting with $\alpha$. With the assistance of Claude
Opus~4.8, I established the desired lower bound on two of the
regions determined by which term of $\alpha$ is largest, using a
cornering argument for the largest square.
The remaining regions resisted this approach, and I paused work
on square packings in early August to work on other projects.
On October 7, 2026, I supplied GPT-6 Astra Pro (OpenAI) with the
candidate formula and the seven-square proof and asked it to
develop lower bounds using those methods. Over a sequence of
exchanges, it developed and assembled proofs of the eight new
inequalities and their arithmetic combination that, after
revision, are presented in Sections~\ref{sec:lower}
and~\ref{sec:proof}. Subsequent exchanges led to the
convex-minimality argument and its supporting certificates.

The manuscript was written and revised using Claude Opus~5.5
(Anthropic) and GPT-6 Astra Pro, which also reviewed the
mathematical arguments. Claude Opus~5.5 prepared the
verification code~\cite{BadgerCode} and used it to check the constructions, figures, and
certificates in exact arithmetic. I have checked all proofs and
am responsible for the content of this paper.

\newcommand{\etalchar}[1]{$^{#1}$}
\providecommand{\MR}{\relax\ifhmode\unskip\space\fi MR }
\providecommand{\MRhref}[2]{%
  \href{http://www.ams.org/mathscinet-getitem?mr=#1}{#2}
}
\providecommand{\href}[2]{#2}


\begin{thebibliography}{{Que}26b}

\bibitem[Bad26a]{BadgerSeven}
Matthew Badger, \emph{Axis-parallel square packing with an exact formula for
  optimal packing of seven squares}, preprint, \texttt{arXiv:2610.22622}, 2026.

\bibitem[Bad26b]{BadgerCode}
Matthew Badger, \emph{paper-code: code accompanying research papers by {M}atthew
  {B}adger}, GitHub repository,
  \url{https://github.com/matthew-badger/paper-code}, folder
  \texttt{eight-squares}, 2026, Verification script version 1.1.0, October 20,
  2026, archived at Zenodo,
  \href{https://doi.org/10.5281/zenodo.23504589}{doi:10.5281/zenodo.23504589}.

\bibitem[EG75]{ErdosGraham75}
Paul Erd{\H o}s and Ronald~L. Graham, \emph{On packing squares with equal
  squares}, Journal of Combinatorial Theory, Series A \textbf{19} (1975),
  119--123.

\bibitem[Fri09]{Friedman09}
Erich Friedman, \emph{Packing unit squares in squares: A survey and new
  results}, Electronic Journal of Combinatorics (2009), Dynamic Survey DS7,
  version of August 14, 2009.

\bibitem[FS04]{FeketeSchepers04}
S\'andor~P. Fekete and J\"org Schepers, \emph{A combinatorial characterization
  of higher-dimensional orthogonal packing}, Mathematics of Operations Research
  \textbf{29} (2004), no.~2, 353--368.

\bibitem[HK13]{HuangKorf13}
Eric Huang and Richard~E. Korf, \emph{Optimal rectangle packing: An absolute
  placement approach}, Journal of Artificial Intelligence Research \textbf{46}
  (2013), 47--87.

\bibitem[KL82]{KarpLipton82}
Richard~M. Karp and Richard~J. Lipton, \emph{Turing machines that take advice},
  L'Enseignement Math\'ematique. Revue Internationale. 2e S\'erie \textbf{28}
  (1982), no.~3-4, 191--209.

\bibitem[LTW{\etalchar{+}}90]{Leung90}
Joseph Y.-T. Leung, Tommy~W. Tam, C.~S. Wong, Gilbert~H. Young, and Francis
  Y.~L. Chin, \emph{Packing squares into a square}, Journal of Parallel and
  Distributed Computing \textbf{10} (1990), no.~3, 271--275.

\bibitem[MFNK96]{Murata96}
Hiroshi Murata, Kunihiro Fujiyoshi, Shigetoshi Nakatake, and Yoji Kajitani,
  \emph{{VLSI} module placement based on rectangle-packing by the
  sequence-pair}, IEEE Transactions on Computer-Aided Design of Integrated
  Circuits and Systems \textbf{15} (1996), no.~12, 1518--1524.

\bibitem[MM67]{MoonMoser67}
J.~W. Moon and L.~Moser, \emph{Some packing and covering theorems}, Colloquium
  Mathematicum \textbf{17} (1967), 103--110.

\bibitem[{Que}26a]{ElevenSquaresFormalized26}
{Queuingtheorydotcom}, \emph{11{S}quares{F}ormalized: {L}ean formalization of
  the optimality proof of the 11 square packing}, GitHub repository,
  \url{https://github.com/Queuingtheorydotcom/11SquaresFormalized}, 2026,
  Verification report dated October 6, 2026; accessed October 16, 2026.

\bibitem[{Que}26b]{ElevenSquares26}
\bysame, \emph{11{S}quares{O}ptimal: a proof of the optimality of the 11 square
  packing}, GitHub repository,
  \url{https://github.com/Queuingtheorydotcom/11SquaresOptimal}, 2026, Posted
  September 29, 2026.

\bibitem[Str03]{Stromquist03}
Walter Stromquist, \emph{Packing 10 or 11 unit squares in a square}, Electronic
  Journal of Combinatorics \textbf{10} (2003), Research Paper 8.

\bibitem[SV20]{SilvanusVygen20}
Jannik Silvanus and Jens Vygen, \emph{Few sequence pairs suffice: Representing
  all rectangle placements}, SIAM Journal on Discrete Mathematics \textbf{34}
  (2020), no.~4, 2017--2032.

\end{thebibliography}
\end{document}